\documentclass[]{amsart}
\usepackage{amssymb,amscd,amsthm}
\usepackage{graphicx}
\usepackage{color}
\usepackage[outdir=./]{epstopdf}
\usepackage{epsfig}

\usepackage[T1]{fontenc}
\usepackage[utf8]{inputenc}
\UseRawInputEncoding

\usepackage[english]{babel}
\usepackage{bm}
\usepackage{amsthm,amssymb,amsmath}
\usepackage{enumerate}
\usepackage{mathtools}
\usepackage[export]{adjustbox}

\calclayout

 \newtheorem{thm}{Theorem}[section]
 \newtheorem{thmA}{Theorem}
 \newtheorem{cor}[thm]{Corollary}
 \newtheorem{lemma}[thm]{Lemma}
 \newtheorem{prop}[thm]{Proposition}
 \theoremstyle{definition}
 \newtheorem{defn}[thm]{Definition}
 \newtheorem{exmp}[thm]{Example}
\newtheorem{notation}[thm]{Notation}

 \theoremstyle{remark}
 \newtheorem{rem}[thm]{Remark}
 \newtheorem{afirm}[thm]{Claim}
 \newtheorem{note}[thm]{Notation}
 \newtheorem{assumption}[thm]{Assumption}
 \newtheorem{quest}[thm]{Question}

 \numberwithin{equation}{subsection}
 \newtheorem{ack}{Acknowledgment}

\makeatletter
\newcommand{\proofpart}[2]{%
  \par
  \addvspace{\medskipamount}%
  \noindent\emph{Part #1: #2}\par\nobreak
  \addvspace{\smallskipamount}%
  \@afterheading
}
\makeatother

\newcommand{\cT}{\text{$\mathcal{T}$}}

\newcommand{\cB}{\text{$\mathcal{B}$}}

\newcommand{\cE}{\text{$\mathcal{E}$}}

\newcommand{\FF}{\text{$\mathcal{F}$}}

\newcommand{\CC}{\text{$\mathcal{C}$}}
\newcommand{\cL}{\text{$\mathcal{L}$}}

\newcommand{\cZ}{\text{$\mathcal{Z}$}}

\newcommand{\sg}{\sigma}
\newcommand{\hol}{\operatorname{hol}}

\newcommand{\id}{\operatorname{id}}

\newcommand{\Diff}{\operatorname{Diff}}
\newcommand{\Homeo}{\operatorname{Homeo}}

\newcommand{\PSL}{\operatorname{PSL}}
\newcommand{\Ker}{\operatorname{Ker}}

\newcommand{\Stab}{\operatorname{Stab}}
\newcommand{\Hol}{\operatorname{Hol}}
\newcommand{\St}{\operatorname{St}}
\newcommand{\Intr}{\operatorname{int}}

\newcommand{\eu}{\operatorname{eu}}
\newcommand{\rot}{\operatorname{rot}}
\newcommand{\lvl}{\operatorname{Lvl}}

        \newcommand{\field}[1]{\text{$\mathbb{#1}$}}
        \newcommand{\N}{\field{N}}
        \newcommand{\Z}{\field{Z}}
        
        \newcommand{\R}{\field{R}}
        \newcommand{\C}{\field{C}}
        \newcommand{\HH}{\mathbb{H}}

\newdimen\theight
\def\TeXref#1{%
             \leavevmode\vadjust{\setbox0=\hbox{{\cT
                     \quad\quad  {\small \textrm #1}}}%
             \theight=\ht0
             \advance\theight by \lineskip
             \kern -\theight \vbox to
             \theight{\rightline{\rlap{\box0}}%
             \vss}%
             }}%

\makeatletter
\def\blfootnote{\gdef\@thefnmark{}\@footnotetext}
\makeatother

\begin{document}

\title{Every noncompact surface is a leaf of a minimal foliation}

\author{Paulo Gusm\~ao$^\dagger$ \& Carlos Meni\~no Cot\'on$^\ddagger$}

\blfootnote{\textup{2010} \textit{Mathematics Subject Classification}: Primary 37C85, 57R30}

\maketitle
	
\address{$\dagger$ Departamento de An\'alise,  Universidade Federal Fluminense,\\ Instituto de Matem\'atica e Estat\'istica, Rua Professor Marcos Waldemar de Freitas Reis  S/N, Campus Gragoatá, CEP 21941-916, RJ, Brazil.}\\

\address{$\ddagger$ CITMAGA \&
Departamento de Matem\'atica Aplicada I, Universidade de Vigo\\ 
Instituto Innovacións Tecnolóxicas, CP 15782 Santiago de Compostela \& Escola de Enxe\~ner\'ia Industrial, Rua Conde de Torrecedeira 86, CP 36208, Vigo, Spain.}\\

\email{email: phcgusmao@id.uff.br \& carlos.menino@uvigo.es}

\begin{abstract}
We show that any noncompact oriented surface is homeomorphic to the leaf of a minimal foliation of a closed $3$-manifold. These foliations are (or are covered by) suspensions of continuous minimal actions of surface groups on the circle. Moreover, the above result is also true for any prescription of a countable family of topologies of noncompact surfaces: they can coexist in the same minimal foliation. All the given examples are hyperbolic foliations, meaning that they admit a leafwise Riemannian metric of constant negative curvature. Many oriented Seifert manifolds with a fibered incompressible torus and whose associated orbifold is hyperbolic admit minimal foliations as above. The given examples are not transversely $C^2$-smoothable.
\end{abstract}

\section*{Introduction}
It is well known that every surface is homeomorphic to a leaf of a $C^\infty$ codimension one foliation on a closed $3$-manifold \cite{Cantwell-Conlon-1987}, these leaves appear at the infinite level of those foliations and they are represented as a Hausdorff limit of finite level leaves (see \cite[Chapter 5]{Candel-Conlon-I-2000} for an overview on the theory of levels). Every foliation with more than one level cannot be minimal, this opens the question about which noncompact surfaces can be homeomorphic to leaves of a minimal foliation on some closed $3$-manifold. 

This question is related to the so-called ``Realization Problem'' that studies what kind of manifolds can be homeomorphic to leaves of foliations. Recall that there exists manifolds of dimension greater than $3$ which are not homeomorphic to any leaf of any  codimension one foliation (see for instance \cite{Ghys-NonLeaf}).

Candel's Uniformization Theorem~\cite{Candel1} implies that a codimension one oriented foliation on a $3$-manifold with no transverse invariant measure (this is the usual case) admits a leafwise hyperbolic metric varying continuously in the ambient space. Thus, it is completely natural to restrict the previous question to minimal hyperbolic foliations.

A noncompact surface (without boundary) is said to have {\em finite type\/} if it has finitely many ends and all of them are planar, i.e., they admit neighborhoods without genus. Equivalently, a noncompact surface of finite type is homeomorphic to a closed surface punctured along finitely many points. A surface of non-finite type is said to be of {\em infinite type}.

There exist several conditions implying that all the leaves of a minimal hyperbolic foliation on a closed $3$ manifold must have the same type, for instance: every foliation on a Sol-manifold, center-stable leaves of transitive Anosov flows or those minimal hyperbolic foliations where each leaf has a finitely generated holonomy group \cite{ADMV,ADMV2}. In addition, if a minimal hyperbolic foliation on a closed $3$-manifold admits a transverse invariant measure, then all the leaves are simply connected or none of them are \cite[Remark 2.4]{ABMPW}.

The study of minimal hyperbolic foliations is extensively treated in \cite{ADMV2} but in each one of the explicit examples of that work where the leaves are of finite type, the leaves are just planes or cylinders. 

These results are the motivations to this work, whose main result is the following:

\begin{thmA}\label{Main Theorem}
Let $O$ be a typical orbifold. There exists a compact Seifert oriented $3$-manifold $M$ whose base orbifold is $O$ such that, for every countable family of noncompact oriented surfaces $S_n$, $n\in\N$, there exists a transversely bi-Lipschitz minimal hyperbolic foliation $\FF$ on $M$, transverse to its fibers, and such that for each $n\in\N$ there exists a leaf $L_n\in\FF$ homeomorphic to $S_n$.
\end{thmA}

These are, as far as we know, the first examples of minimal hyperbolic foliations on some closed $3$-manifold where leaves of finite and infinite type do coexist.

An orbifold $O$ will be called {\em typical} if its underlying space admits a simple separating loop so that each connected component of the complement of this loop inherits an orbifold structure from $O$ that admits a hyperbolic structure of finite area (every end is a cusp). As a consequence, $O$ also admits a hyperbolic structure. For instance, every closed surface of genus $g\geq 2$ is a typical orbifold.

For a Seifert oriented manifold with oriented base of genus $g$, let $$\{(o_1,g);b;(a_1,b_1),\dots,(a_k,b_k)\}$$ be a symbolic presentation. Here $o_1$ means that both the ambient $3$-manifold and the base orbifold are orientable, $b$ is an integer representing an ordinary fiber  with a tubular neighborhood fibered by circles performing $b$ turns around it. The pairs $(a_i,b_i)$ satisfy $0<b_i<a_i$, with $a_i,b_i\in \N$, they are structure constants of the exceptional fibers and represent tubular neighborhoods fibered by circles performing $2\pi b_i/a_i$ rotations around the respective exceptional fiber.

The {\em (rational) Euler number} of the  Seifert $3$-manifold $M$ is defined as $\eu(M)=b+\sum b_i/a_i$ and it is an invariant of the Seifert manifold. Selberg's Theorem \cite{Selberg} implies that all of our examples are finitely covered by suspensions of a surface group action; this finite cover is a trivial fiber bundle if and only if the Euler number is zero.

Since the ambient manifold of our examples is always a Seifert manifold where the leaves are transverse to the fibers, they are easily shown to be also $\R$-covered, i.e. the leaf space of the lifted foliation to the universal covering space is homeomorphic to $\R$ (see e.g. \cite{Fenley}). 

We want to remark that in a recent work \cite{ABMPW}, S. \'Alvarez, J. Brum, M. Mart\'inez and R. Potrie provide an interesting construction of a minimal hyperbolic lamination on a compact space where all the noncompact oriented surfaces are realized (topologically) as leaves. They also show that this lamination can be embedded in $\C P^3$, but it cannot be realized as a minimal set of a codimension one foliation. The techniques used in \cite{ABMPW} are completely different from the ones given in this work.

Note that not every Seifert oriented manifold with a typical base admits minimal foliations such as those given by Theorem~\ref{Main Theorem}, this is a consequence of  \cite[Teor\'em\'e 3]{Ghys2}: let $\Sigma$ be a closed surface with genus $\geq 2$ and let $\phi:\pi_1(\Sigma)\to \Homeo_+(S^1)$ be a homomorphism, let $M$ be the Seifert manifold obtained by the quotient of the diagonal action of $\pi_1(\Sigma)$ on $\widetilde{M}\times S^1$ via $\phi$ and assume that $|\eu(M)|= 2g-2$ (i.e., it is maximal\footnote{The Euler number of this Seifert manifold is the same as the Euler number of the oriented flat bundle defined by $\phi$ over $\Sigma$, hence it satisfies the Milnor-Wood inequality.}), then each leaf of the suspended foliation (see Subsection~\ref{ss:suspensions} for details) must be a plane, a cylinder or a torus. Of course, the last option does not occur if the foliation is minimal.

Our method of construction allows us to produce examples of minimal foliations with nonplanar and noncylindrical leaves on ambient Seifert fibrations over hyperbolic surfaces with Euler number (in absolute value) $2g-4k$, with $k$ ranging between $1$ and the integer part of $g/2$. Particular cases of the above Theorem are those where $M$ is the product of $S^1$ with any hyperbolic closed surface of even genus. Our examples are suitable perturbations of rather explicit group actions obtained by amalgamation of Fuchsian ones, so we expect that the construction can be improved to obtain more ambient $3$-manifolds where Theorem~\ref{Main Theorem} also works.

It is worth noting that the transverse regularity of our examples is bi-Lipschitz (and moreover differentiable) but we were not able to reach transverse regularity $C^1$. We conjecture that this regularity should be achieved with some precise modifications of our construction. Regularities higher than $C^1$ cannot be obtained with our present construction and seem out of the reach at this point.

From the point of view of the ``Realization Problem'' it is also interesting to understand which Riemannian manifolds can be realized as leaves of foliations; bounded geometry is an obvious obstruction but, in general, there exist Riemannian manifolds with bounded geometry that are not (quasi)-isometric to any leaf of any codimension one foliation (see \cite{Hurder1993} for a detailed study of the coarse geometry of foliations). The quasi-isometry types of the leaves in our examples come from lifts over typical closed orbifolds, but it is unclear if every possible lift can be so realized.

For the sake of readability, now we will provide a sketch of the topological realization of a punctured torus as a leaf in a minimal $C^0$ foliation on the manifold $M= S^1\times S$, where $S$ is a closed surface of genus $2$.

First note that $\pi_1(S)$ is described as an amalgamated product of two free groups representing the fundamental groups of two handles, namely $H^\triangleright$ and $H^\triangleleft$, attached by their boundaries. Choose a hyperbolic metric of finite area on a  once punctured torus (ie, the interior of a handle), its fundamental group acts by isometries in the hyperbolic plane and can be represented in $\PSL(2,\R)$ by the induced action in the circle at infinity. The suspensions of these specific representations over the handles $H^\triangleright$ and $H^\triangleleft$, will be called {\em projective foliated blocks} and denoted $\FF^\triangleright$ and $\FF^\triangleleft$. They are minimal by construction. The leaf topology of these foliated blocks is well understood, all but countably many leaves are homeomorphic to the universal covering space of a handle (see Proposition~\ref{p:3 topologies}). Observe that each foliated block has a transverse torus whose trace foliation comes from the suspension of a parabolic circle diffeomorphism and therefore the trace orbits are lines spiraling towards a single closed orbit, these orbits are identified with the boundary components of leaves of the foliated block. The construction of foliated blocks can be made in any typical orbifold and this is the reason why we deal with this generality in Section~\ref{s:projective blocks}.

The second main point is gluing back the fundamental blocks $\FF^\triangleright$ and $\FF^\triangleleft$ in a suitable way in order to guarantee the existence of a leaf homeomorphic to the chosen one. In section \ref{s:transverse gluing} it is shown that gluing maps between the trace foliations of the boundary tori of the foliated blocks are bijective with a suitable space of interval homeomorphisms. Therefore the family o gluing maps forms a Baire space. In the third section we show several generic properties of these gluing maps. Observe that for any gluing map the resulting foliation is still minimal since any foliated block is also minimal.

The most important concept now is that of ``predefined stabilizer''. Topologically, it should be understood as ``predefined junctions of leaves''. Instead of working with the whole space of gluing maps, we restrict to those maps where the images of some discrete set of points are defined ``{\em a priori}'' (these points are identified with boundary components of the leaves of the foliated blocks). For instance, if we want to create a handle in a leaf, we can predefine the image of just two points in order to reproduce the construction given in Figure~\ref{f:H1} in the leaf passing through those points.

Now, a Baire type argument implies that for a generic choice of gluing maps with the same predefined images of points no more topologies arise (generically we are always attaching new generic leaves of the foliated blocks to the predefined construction and this does not produce other stabilizers different from the predefined ones). As a consequence the leaf passing through the predefined points is homeomorphic to a punctured torus as desired.

This process can be continued inductively in order to realize any oriented noncompact surface and, moreover, any countable family of surfaces as stated in the Main Theorem~\ref{Main Theorem}. This is done in Section~\ref{s:realizing topology}.

\begin{ack}
We want to thank the Uruguayan team formed by professors S. \'Alvarez, J. Brum, M. Martinez and R. Potrie (Universidad Federal de la Rep\'ublica - Uruguay) for their interest, ideas and stimulating talks at Rio and Montevideo. We also want to thank Ministerio Ciencia y Educaci\'on (Uruguay) for its support during the research visit of the second author to that institution. The second author also wants to thank the Programa Nosso Cientista Estado do Rio de Janeiro 2015-2018 FAPERJ-Brazil (``Din\^amicas n\~ao hiperb\'olicas''), MathAmSud 2019-2020 CAPES-Brazil (``Rigidity and Geometric Structures on Dynamics''), the CNPq research grant 310915/2019-8 and the Ministerio de Ciencia e Innovaci\'on (Spain) grant PID2020-114474GB-I00 that partially supported this research. 
\end{ack}


\section{Projective foliated blocks}\label{s:projective blocks}

\subsection{Hyperbolic orbifolds with cusps}
For basic definitions and results relative to orbifolds we refer to \cite[Chapter 13]{Thurston} and \cite[Chapters 2, 4]{Montesinos}. Informally, an orbifold is a topological space (called the {\em underlying space}) which admits an open covering such that any open set of the cover is the quotient of an euclidean ball by a finite group action, the changes of coordinates beetween these ``charts'' must be coherent with the given structure. The orbifold is orientable if there is a global orientation on the lifted charts which is preserved by coordinate changes.

The underlying space of a $2$-dimensional oriented closed (compact with no boundary) orbifold is just a compact surface with finitely many distinguished points where the isotropy group is non trivial, a neighborhood of these points looks like a cone which is the quotient of an open $2$-ball by a rational rotation, these form the {\em singular set} of the orbifold.

Two orbifolds are {\em isomorphic} if their underlying spaces are homomeomorphic by a homeomorphism that preserves the singular set and is compatible with the orbifold structure, i.e., the isotropy groups of identified singular points are isomorphic and the homeomorphism can be locally lifted on each chart in an equivariant way.

It is said that a $2$-dimensional orbifold $O$ admits a {\em hyperbolic structure} if there exists a properly discontinuous action of a group $\Gamma$ of isometries of the hyperbolic plane $\HH$ such that $O$ is isomorphic to the orbifold whose underlying space is $\HH/\Gamma$ and the structure group of each singular point is isomorphic to the isotropy group of the $\Gamma$ action on any of its representatives in $\HH$, it will be said that $O$ is a {\em hyperbolic orbifold}. This $\Gamma$ is also defined as the fundamental group of the orbifold.  Observe that a hyperbolic orbifold may admit different hyperbolic structures in the sense that the corresponding group actions does not need to be conjugated in $\PSL(2,\R)$.

Most of the $2$-dimensional closed orbifolds are hyperbolic. In our specific case where the orbifolds are oriented and the singular locus is given by cones the unique non-hyperbolic orbifolds are $S^2$, $S^2(n)$, $S^2(n,m)$, $S^2(n,2,2)$, $S^2(3,3,2)$, $S^2(5,3,2)$, $S^2(4,3,2)$, $S^2(4,4,2)$, $S^2(3,3,3)$, $S^2(6,3,2)$, $S^2(2,2,2,2)$ and $T^2$ (see e.g. \cite[Theorem 13.3.6]{Thurston}). Here the notation $M(n_1,\dots, n_k)$ refers to a closed surface $M$ (the underlying space) with $k$ cone points, each one with structure group $\Z_{n_i}$.

Let $\bm{O}$ denote the underlying space of an (oriented) orbifold $O$ and let $\sigma: S^1\to \bm{O}$ be a loop that does not meet any singular point. Let $O^{\Join}$ be the (noncompact) orbifold whose underlying space is $\bm{O}\setminus \bm{\sigma}$ (with the same singular points as $O$), where $\bm{\sigma}=\sigma(S^1)$.  Clearly $\bm{O}^{\Join}$ has two ends corresponding to each side of $\sigma$; if $\bm{O}^{\Join}$ is not connected, then its two noncompact connected components will be denoted as  $\bm{O}^\triangleright$ and $\bm{O}^\triangleleft$.

The following Proposition is a direct corollary of classical theory of Fuchsian groups (see e.g. \cite[Chapter 4]{Katok}).

\begin{prop}\label{p:hyperbolic structure}
Let $O$ be a $2$-dimensional orbifold whose singular set consists of a finite number of cone points and its underlying space admits a finite (nonzero) number of punctures. Then $O$ admits a hyperbolic structure and its fundamental group $\pi_1(O)$ is isomorphic to a free product of cyclic groups. Each generator with finite order is in correspondence with a cone point and its order is the same as the order of the corresponding structure group. Moreover, $O$ admits a hyperbolic structure where every end is a cusp (finite area) if and only if $\pi_1(O)$ is not virtually cyclic {\em (}$\Z,\Z_2\ast\Z_2$ or finite{\em )}.
\end{prop}

\begin{rem}\label{r:minimal cusps}
The fundamental group of a hyperbolic orbifold with finite area acts on $S^1$ by projective diffeomorphisms. It is well known that these actions are minimal (see e.g. \cite[Theorems 4.5.1 and 4.5.2]{Katok}).
\end{rem}

\begin{defn}\label{d:typical orbifold}
It is said that a hyperbolic closed orbifold $O$ is {\em typical} if there exists a simple loop $\sg:S^1\to \bm{O}$ which does not meet its singular locus, $\bm{O}\setminus \bm{\sg}$ is not connected and each connected component admits a hyperbolic structure with finite area\footnote{ More precisely, the Dirichlet domain for the action of the fundamental group of the orbifold on $\HH$ has finite area.}. The loop $\sg$ is said to be a {\em separating} loop. 
\end{defn}

\begin{rem}
According to the classification of hyperbolic closed $2$-dimensional orbifolds (see e.g. \cite{Thurston}) and Proposition~\ref{p:hyperbolic structure} it follows that the unique closed hyperbolic orbifolds that are non-typical are those of the form $S^2(n,m,k)$, $S^2(2,2,2,n)$,  $S^2(2,2,2,2,2)$, $T^2(n)$ and $T^2(2,2)$ for any admissible values of $n,m,k\geq 2$.
\end{rem}

\begin{notation}\label{n:block data}Let $O$ be a typical orbifold and let $\sigma$ be a separating loop for its underlying space. Let $G^\triangleright$ and $G^\triangleleft$ be subgroups of $\PSL(2,\R)$ inducing a hyperbolic structure on $O^\triangleright$ and $O^\triangleleft$ of finite area. Let $B^\triangleright$ and $B^\triangleleft$ be the compact surfaces with boundary obtained from the underlying spaces $\bm{O}^\triangleright$ and $\bm{O}^\triangleleft$ by removing pairwise disjoint open and connected neighborhoods of cone points and ends. Assume that these neighborhoods are bounded by circles.

In order to relax notations, we shall use the symbol $\star$ on super and subindices  to refer simultaneusly to $\triangleright$ and $\triangleleft$. Observe that $B^\star$ is canonically oriented and its orientation is induced by an orientation in $O$. Let $\alpha_{1\star},\dots,\alpha_{k\star}$ be parametrizations of the boundary components of $B^\star$ associated to cone points of $O^\star$ and let $\beta_\star$ be a parametrization of the boundary component associated to the end of $O^\star$. Choose the parametrizations following the orientation induced by $B^\star$ in its boundary. This implies that the orientation of $\beta_\triangleright$ is opposed to that of $\beta_\triangleleft$.
\end{notation}

Let $[\alpha]$ denote the homotopy class of $\alpha$ relative to some distinguished point in $B^\star$\footnote{This requires the use of a path $\gamma$ from the distinguished point to $\alpha$; thus $[\alpha]$ refers to the homotopy class of the junction of $\gamma$, $\alpha$ and $\gamma^{-1}$, this homotopy class depends on the choice of $\gamma$ but its conjugacy class does not. We make this abuse of notation for the sake of readability.} Observe that $\pi_1(B^\star)$ is a free group and it admits a system of generators containing $[\alpha_{1\star}],\dots,[\alpha_{k\star}]$. There exists a surjective homomorphism $h^\star:\pi_1(B^\star)\to \pi_1(O^\star)\equiv G^\star$ and $\Ker(h^\star)$ is the normal closure of $\langle\{[\alpha_{1\star}]^{n_1},\dots,[\alpha_{k\star}]^{n_k}\}\rangle$, where $n_i$ is the order of $h^\star([\alpha_{i\star}])$.  Observe that $G^\star$ is a free product of cyclic groups and $h^\star([\alpha_{i\star}])$ is defined as an elliptic generator of $G^\star$ associated to the cone point surrounded by $\alpha_{i\star}$. The other generators can be mapped in a natural way with the generators of $G^\star$ of infinite order. The boundary component associated to the cusp, the trace of $\beta_\star$, is clearly homotopically equivalent to a suitable word on the generators given above (since it can be retracted to a $1$-complex formed by a wedge of loops homotopic to them), this is, in fact, the classical presentation of a Fuchsian group. Recall that $h^\star(\pi_1(B^\star))$ is isomorphic to $\pi_1(O^\star)$ and it is naturally included in $\PSL(2,\R)$ and hence in $\Diff_+^\omega(S^1)$.

\begin{defn}\label{d:foliated block}
Let $\FF(G^\star,B^\star,h^\star)$ denote the suspension of the previously defined homomorphism $h^\star$.
\end{defn}

From Remark~\ref{r:minimal cusps} it follows that $\FF(G^\star,B^\star,h^\star)$ is a minimal foliation on a compact $3$-manifold with boundary. The boundary of this foliation consists of finitely many tori that are in correspondence with the boundary components of $B^\star$. The trace foliation in each of these tori is given by the suspension of the circle diffeomorphisms $h^\star(\alpha_{i\star})$'s and $h^\star(\beta_\star)$.  This allows to classify the boundary components of $\FF(G^\star,B^\star,h^\star)$ as {\em elliptic} or {\em parabolic} accordingly with its trace foliation.

\begin{prop}\label{p:always parabolic}
For every $i\in\{1,\dots k\}$, the diffeomorphism $h^\star([\alpha_{i\star}])$ is elliptic. The diffeomorphisms $h^\star([\beta_\star])$ are always parabolic. Without loss of generality we can also assume that $h^\triangleright([\beta_\triangleright])$ and $h^\triangleleft([\beta_\triangleleft])^{-1}$ belong to the same conjugacy class in $\PSL(2,\R)$.
\end{prop}
\begin{proof}
The first affirmation is trivial, finite order projective diffeomorphisms are always elliptic. The maps $h^\star(\beta_\star)$ are parabolic by construction since they fix just one point in the circle (corresponding to the cusp). It follows that $h^\triangleright([\beta_\triangleright])$ is conjugated to $h^\triangleleft([\beta_\triangleleft])$ or its inverse, since, in $\PSL(2,\R)$ parabolic maps have just two conjugacy classes. If the definition of $h^\triangleleft$ on the generators (of the previous presentation of $\pi_1(B^\triangleleft)$ as a free group) is replaced by their inverses, then the conjugacy class of $h^\triangleleft([\beta_\triangleleft])$ is changed. Thus, $h^\star$ can be chosen to guarantee that $h^\triangleright([\beta_\triangleright])$ and $h^\triangleleft([\beta_\triangleleft])^{-1}$ belong to the same conjugacy class in $\PSL(2,\R)$ as desired.
\end{proof}

\begin{rem}[Dehn fillings]\label{r:Dehn filling}
Let $T_i$ be the elliptic boundary torus of $\FF(G^\star,B^\star,h^\star)$ whose trace foliation is the suspension of the elliptic element $h^\star([\alpha_{i\star}])$ of order $n_i$. The trace foliation is conjugated to a Kronecker flow associated to a rotation with rotation number $\rot(h^\star([\alpha_{i\star}]))$. Let $\FF_\times$ be the profuct foliation in $D^2\times S^1$. It follows that the boundary elliptic torus $T_i$ can be glued with the boundary of $\FF_\times$ via a suitable Dehn map. Since the trace foliations are preserved by this map, the resulting foliation $\FF(G^\star,B^\star,h^\star)\cup_\partial \FF_\times$ has one elliptic boundary torus less than the former $\FF(G^\star,B^\star,h^\star)$. This process is called {\em Dehn's filling}.
\end{rem}

\begin{defn}\label{d:foliated block Dehn}
Let $\overline{\FF}(G^\star,B^\star,h^\star)$ be the foliation obtained by performing Dehn's fillings on each elliptic boundary torus of $\FF(G^\star,B^\star,h^\star)$. These foliations will be called {\em foliated (projective) blocks}.
\end{defn}

\begin{rem}
Recall that $\FF(G^\star,B^\star, h^\star)$ has a natural structure of Seifert fibration given by the vertical fibers of $\widetilde{B^\star}\times S^1$. The Dehn filling performed on the elliptic torus $T_i$ maps the fibers, which form a product fibration on $T_i$, with the fibers at the boundary of a Seifert fibration on the solid torus $\FF_\times$ of type $\rot(h^\star([\alpha_{i\star}]))$. Of course, the trace leaves in $T_i$ are mapped to the boundary of the leaves of the product foliation on $\FF_\times$. 

It follows that the ambient manifold of $\overline{\FF}(G^\star,B^\star,h^\star)$ is also a Seifert manifold with boundary. 
\end{rem}

\subsection{Topology of leaves of suspensions}\label{ss:suspensions}\label{ss:suspension}
Just for completeness we recall the relation between the holonomy and fundamental groups of a leaf of a suspension foliation.

Let $h:\pi_1(M,x)\to \Homeo(T)$ be a representation of the fundamental group of a manifold $M$, with a distinguished point $x$, into the group of homeomorphisms of a topological space $T$. The suspension foliation $\FF_h$ is obtained as the quotient of $\widetilde{M}\times T$ (foliated  by $\widetilde{M}\times\{\cdot\}$), by the diagonal action $\gamma\cdot(z,t) = (z\cdot \gamma,h(\gamma)^{-1}(t))$ (here $z\cdot\gamma$ denotes the right action of $\gamma$ as deck transformation). Given $z\in\widetilde{M}$ and $t\in T$, let $[(z,t)]$ be the equivalence class of $(z,t)$ by the diagonal action and let $L_{(z,t)}$ denote the leaf of $\FF_h$ that meets $[(z,t)]$.

\begin{defn}\label{d:stabilizer_group}
Let $\rho:\Gamma\to\Homeo(T)$ be a homomorphism from a group $\Gamma$ to the group of homeomorphisms of a topological space $T$. Let $t\in T$ and let us define the group of stabilizers of $\rho$ at $t$ as $\Stab_\rho(t)=\{g\in \Gamma\mid\ \rho(g)^{-1}(t)=t\}$.
\end{defn}

It is well known that $L_{(x,t)}$ is a covering space of $M$ and its fundamental group is isomorphic to $\{\gamma\in\pi_1(M,x)\mid\ h(\gamma)^{-1}(t)=t\}=\Stab_h(t)$. \\

Let $\Hol(L_{(x,t)})$ be the holonomy group of the leaf $L_{(x,t)}$, which is a group of germs of homeomorphisms of $T$ at $t$ and let $\hol:\Stab_h(t)\subset\pi_1(M,x)\to \Hol(L_{(x,t)})$ be the holonomy representation of the fundamental group of $L_{(x,t)}$, $\hol(\gamma)$ is just the germ of $h(\gamma)$ at $t$ . Note that $\hol(\gamma)$ can be trivial (a germ of the identity at $t$) even when $h(\gamma)$ is not the identity. More precisely, there exists a surjective morphism $p_{t}:h(\Stab_h(t))\to\Hol(L_{(x,t)})$.

\begin{defn}
A morphism $h:\pi_1(M,x)\to\Homeo(T)$ is called {\em locally faithful} if the maps $p_{t}:h(\Stab_h(t))\to\Hol(L_{(x,t)})$ are isomorphisms for all $t\in T$. The corresponding associated action of $\pi_1(M,x)$ on $T$ is also called {\em locally faithful}.
\end{defn}

The simplest example of locally faithful actions are those given by faithful actions of analytic diffeomorphisms. Observe that even when $h$ is not a monomorphism, the group $h(\pi_1(M,x))$ is a subgroup of $\Homeo(T)$ and therefore $h(\pi_1(M,x))$ acts in a faithful way on $T$. The previous discussion is summarized in the following Lemma:

\begin{lemma}\label{l:leaf topology suspension}
Let $h:\pi_1(M,x)\to\Homeo(T)$ be a homomorphism. Then the fundamental group of a leaf $L_{(x,t)}$ is isomorphic to $\Stab_h(t)$. If $h$ is locally faithful, then the holonomy group of $L_{(x,t)}$ is isomorphic to $h(\Stab_h(t))$. If, moreover, $h$ is a monomorphism then $\Stab_h(t)$ and $\Hol(L_{(x,t)})$ will be isomorphic.
\end{lemma}

\subsection{Foliated projective blocks}
From the Definitions~\ref{d:foliated block} and \ref{d:foliated block Dehn}, given a typical closed orbifold $O$ and separating loop $\sg$, we get a pair of minimal foliated projective blocks $\overline{\FF}(G^\star,B^\star,h^\star)$ with transverse boundary consisting of parabolic tori. We shall always assume that the separating loop $\sg$ and the hyperbolic structures of finite area in $O^\triangleright$ and $O^{\triangleleft}$ have been chosen and fixed. The associated foliated projective blocks over $B^\triangleright$ and $B^\triangleleft$ will be denoted by $\FF^\triangleright$ and $\FF^\triangleleft$, respectively.

Since the foliated projective blocks are minimal, every leaf is noncompact and has free fundamental group. It is well known that the stabilizer group of a projective discrete group at any point is always solvable, it follows, from Lemma~\ref{l:leaf topology suspension}, that the holonomy group of each leaf must be infinite cyclic or trivial. This is a particular version of the so called Hector's Lemma (see e.g. \cite[Proposition 3.7]{Matsuda}) which says that the stabilizer of any point in a minimal locally discrete analytic group action on the circle must be trivial or infinite cyclic. If moreover $h$ is a monomorphism, it follows that the fundamental group is also trivial or infinite cyclic.

\begin{prop}\label{p:3 topologies}
Let $\overline{\FF}(G^\star,B^\star,h^\star)$ be a foliated projective block, then
\begin{enumerate}
\setcounter{enumi}{-1}
\item All but countably many leaves are simply connected, all of them are homeomorphic to an oriented surface with infinitely many non-compact boundary components and a Cantor set of ends.\label{Type 0}

\item There exists countably many leaves homeomorphic to an oriented surface with infinitely many noncompact boundary components, cyclic fundamental group and a Cantor set of ends.\label{Type 1}

\item There exists exactly one leaf homeomorphic to an oriented surface with infinitely many boundary components where only one of these components is a circle, with cyclic fundamental group and a Cantor set of ends. \label{Type 2}
\end{enumerate}

\end{prop}
\begin{proof}
Set $\FF^\star=\overline{\FF}(G^\star,B^\star,h^\star)$ and consider first the suspension $\FF(G^\star,B^\star,h^\star)$. Since the orbifold is typical it follows that $G^\star$ is not virtually cyclic. Therefore $G^\star$ has a Cantor set of ends. Observe also that the leaves are oriented since $G^\star$ preserves orientation (it is a projective action).

Assume first that the base orbifold is a surface, i.e., there are no cone points and therefore no elliptic boundary components. In this case $G^\star$ is a free group and the action  defined by $h^\star$ is faithful.  Since the action is also analytic, it follows that all but a countable set of leaves of $\FF(G^\star,B^\star,h^\star)$ are homeomorphic to the universal covering of $B^\star$, so they are simply connected,  with no compact boundary components and have a Cantor set of ends. Since the foliation is minimal, the boundary of simply connected leaves must meet the (parabolic) boundary torus in a dense family of infinitely many connected components that are orbits of the respective trace foliations  associated to $h^\star([\beta_\star])$. This is the case \ref{Type 0} of the proof.

The leaves which are not simply connected are those corresponding to those points which are fixed by some element of $G^\star$. Let $x\in S^1$ be a fixed point of some $g\in G^\star\neq{\id}$. Since the holonomy of every leaf is cyclic, it follows that $\Hol(L_x,x)=\langle g\rangle$ and $\pi_1(L_x,x)\approx \Z$, where $L_x$ is the leaf containing $x$ (here $g$ defines just a germ in a transverse neighborhood of $x$, but there is a unique analytic continuation). Observe also that $f(x)$ is fixed by $fgf^{-1}$ for all $f\in G^\star$, and that point belongs to the $G^\star$-orbit of $x$. Since the action is faithful, $G^\star$ does not contain elliptic elements, thus $fgf^{-1}$ is parabolic or hyperbolic for all $f\in G^\star$, as a consequence it has at most two fixed points. It follows that there exists a correpondence at most $2$ to $1$ between non simply connected leaves of the suspension and conjugacy classes of $G^\star$. Since the conjugacy classes of a free group are infinite it follows that there are infinite, but countably many, leaves which are not simply connected. These leaves are cyclically covered by the universal cover of $B^\star$, so they still have a Cantor set of ends. 

A non simply connected leaf must lie within these two types:
\begin{enumerate}
\item Its boundary consists of noncompact orbits of the trace foliation of the parabolic torus. There are countably many leaves with this property.

\item Its boundary contains the closed orbit of the  trace foliation of the parabolic torus, this is the unique compact boundary component of the leaf and any parametrization of this component generates the fundamental group of the leaf.
\end{enumerate}

In any case, these leaves are dense, so they must have infinitely many noncompact boundary components corresponding to noncompact orbits of the trace foliation of the parabolic torus. The endset of the leaves in the previous cases still are Cantor sets since they are cyclically  covered by $\widetilde{B^\star}$. This completes the remaining cases \ref{Type 1} and \ref{Type 2}.

If the orbifold has cone points, then the foliation $\FF(G^\star,B^\star,h^\star)$ has elliptic boundary tori. The kernel of the action is, by the definition of $h^\star$, the normal closure of $[\alpha_{1\star}^{n_1}],\dots,[\alpha_{k\star}^{n_k}]$ , where each $\alpha_{i\star}$ represents a boundary component of $B^\star$ which bounds a cone point of order $n_i$.

The existence of a nontrivial kernel implies nontrivial topology in every leaf. In this case, the topology appears in the form of infinitely many compact boundary components on each leaf, they correspond with the (closed) orbits of the elliptic tori. Since $\Ker(h^\star)$ is the minimal normal subgroup generated by these elements, no more topology coming from elements in $\Ker(h^\star)$ can appear.

Observe that none of these components can generate nontrivial holonomy since they are associated to elliptic elements. Thus, in any case, the holonomy of every leaf is still trivial or infinite cyclic. Finally, observe that after performing Dehn fillings on the elliptic tori, all the compact boundary components with trivial holonomy on every leaf are filled with a disk and we recover the three topologies described above for leaves in $\FF^\star$.
\end{proof}

The topologies described in Proposition~\ref{p:3 topologies} determines exactly three different noncompact oriented surfaces which are pictured in Figure~\ref{f:3 topologies}.

\begin{figure}
\centering
\includegraphics[scale=0.33,frame]{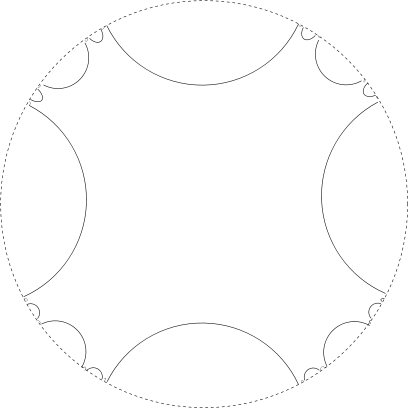}
\includegraphics[scale=0.33,frame]{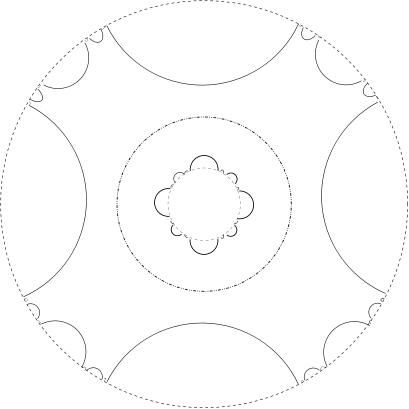}
\includegraphics[scale=0.33,frame]{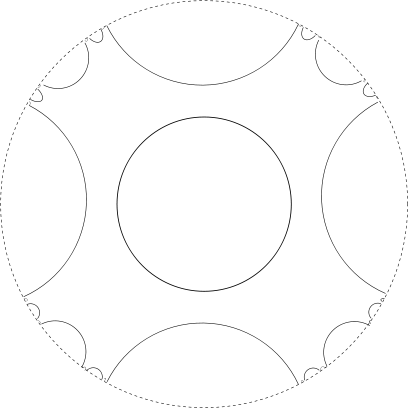}
\caption{Leaves of type $0$, $1$ and $2$ embedded in the hyperbolic plane. Dashed great circles represent the circle at infinity. On the leaf of type $1$, the middle dashed circle represents the homotopy generator. Bold lines are boundary components.}\label{f:3 topologies}
\end{figure}

\begin{defn}\label{d:types of leaves}
As suggested in Proposition~\ref{p:3 topologies}, the simply connected leaves of a foliated block $\FF^\star$ will be called {\em leaves of type} $0$. Leaves with infinite cyclic fundamental group and no compact boundary leaves will be called {\em leaves of type} $1$, the remaining leaf with a compact boundary component will be called {\em leaf of type} $2$. Let us define $\cL^{i}_{\star}$ as the set of leaves in $\FF^\star$ of type $i$, $i\in\{0,1,2\}$. 
\end{defn}

Let $L\in\cL^1_\star$ and let $\tau$ be a loop in $L$ which generates 
its fundamental group; let $\bm{\tau}$ be the trace of that loop. We can assume without loss of generality that $\bm{\tau}$ does not meet $\partial L$. Observe that $L\setminus \bm{\tau}$ has two noncompact connected components. Let us endow each leaf of a projective block $\FF^\star$ with a canonical orientation induced from that of the base orbifold. Relative to this orientation,  each connected component of $L\setminus\bm{\tau}$ induces an orientation in $\bm\tau$.

\begin{defn}\label{d:out in}A connected component of
 $L\setminus\bm{\tau}$ is called  {\em inner} if the orientation induced in $\bm{\tau}$ agrees with the orientation of $\tau$ and it will be called {\em outer} otherwise. The family of the boundary connected components lying in the outer (resp. inner) component of $L\setminus\bm{\tau}$ will be denoted by $C^+_L$  (resp. $C^-_L$). A boundary component of $L$ in $C^+_L$ (resp. $C^-_L$) will be also called {\em outer} (resp. {\em inner}) boundary component.
\end{defn}

The definition of $C^+_L$ or $C^-_L$  does only depend on the orientation of $\tau$ (since any other choice has a trace $\Z_2$-homologous with $\bm{\tau}$).  It will be assumed that an orientation for a separating loop is chosen for every leaf in $\cL^1_\star$.

\begin{lemma}\label{l:dense out inn}
Let $L$ be a leaf of type $1$ of a foliated projective block $\FF^\star=\overline{\FF}(G^\star,B^\star,h^\star)$. Then both $\bigcup_{B\in\CC_L^{+}}B$ and $\bigcup_{B\in\CC_L^-}B$ are dense in $\partial \FF^\star$.
\end{lemma}
\begin{proof}
Let $E^+$ (resp. $E^-$) be the connected component of $L\setminus\bm{\tau}$ that contains $C_L^+$ (resp. $C_L^-$). Let $\{K^+_n\}_{n\in\N}$ be a family of compact subsets of $E^+$ such that $K^+_n\subset K^+_{n+1}$ for all $n\in\N$ and $\bigcup_n K^+_n = E^+_n$. Define the {\em limit set} of $L$ in the ends of $E^+$ as $$\lim_{E^+} L = \bigcap_{n\in\N} \overline{E^+\setminus K^+_n}\;.$$ It is clear that $\lim_{E^+} L$ is a compact and saturated set (i.e., a union of leaves)  of $\FF^\star$ and therefore it contains a minimal set. Since the foliation is minimal it follows that $\lim_{E^+} L $ is the whole ambient manifold and therefore $E^+$ is a dense subset. Recall that $\partial\FF^\star$ is the parabolic torus and the outer boundary components of $L$ are the connected components of $E^+\cap \partial\FF^\star$. By construction, any sufficiently thin neighborhood of $E^+\cap \partial\FF^\ast$ in $L$ is also contained in $E^+$, thus if $x\in E^+$ is sufficiently close to $\partial\FF^\ast$, then there exists $y\in E^+\cap\partial\FF^\ast$ which is also close to $x$. Hence, by means of the global density of $E^+$, it follows that $E^+\cap \partial\FF^\ast$ is dense in $\partial\FF^\ast$. An analogous argument shows that $E^-\cap\partial\FF^\star$ is dense in $\partial\FF^\star$ completing the proof. 

\end{proof}

\begin{rem}[Relative Euler number]\label{r:relative euler class}

Let $E^\star$ be the ambient manifold of a foliated projective block $\FF(\pi_1(O^\star),B^\star,h^\star)$ (see Definitions~\ref{d:foliated block} and \ref{d:foliated block Dehn}). Assume that $O^\star$ has no singular points (thus it is a manifold) and recall that $E^{\star}$ is the total space of a circle bundle over $B^\star$. Since $h(\beta_\star)$ is parabolic and the  trace of $\beta_\star$ is $\partial B^\star$, we can use its fixed point in order to define a section $s_{h_\star}:\partial B^\star\to E^\star$.

The relative Euler number $\eu(h^\star,s_{h_\star})$ measures the obstruction to extend $s_{h_\star}$ to a global section of the circle bundle defined by $h^\star$ over $B^\star$. In general $|\eu(h^\star,s_{h_\star})|\leq |\chi(B^\star)|$ (see \cite[Proposition 3.2]{Ghys2}) but in our foliated blocks $|\eu(h^\star,s_{h_\star})|=|\chi(B^\star)|=2g-1$ that can be computed explicitly by using the Milnor algorithm (observe that this is a Fuchsian action so a maximal relative Euler number is expected). The sign just depends in the chosen orientation in $B^\star$. This is the analogue of the Milnor-Wood inequality \cite{Milnor,Wood} for compact surfaces with boundary.

If $O^\star$ has cone points, then similar results apply. In this case $\eu(h^\star,s_{h_\star})$ is defined as the sum of the rotation numbers of the elliptic elements chosen in the presentation of $G^\star$ and the number obtained by using the Milnor algorithm in the parabolic torus as above. The absolute value of the relative Euler number is in this case the absolute value of the Euler characteristic of the base orbifold (again this comes from the fact that the action is Fuchsian, see \cite[Chapter 5]{Montesinos}).
\end{rem}
 
\subsection{Examples}\label{ss:subsection examples}

\begin{exmp}\label{e:F-H}
Let $S$ be the closed surface of genus $2$ and the separating loop $\sigma$ is the one such that $S\setminus \bm{\sg}$ is homeomorphic to two punctured torus. In this case each $B^\star$ is homeomorphic to $H$, where $H$ is a torus with a disk removed (a handle). See Figure~\ref{f:F-H-P} in order to check the elements of this example.

Recall that $\partial H$ is a circle and $\pi_1(H)=\Z\ast\Z=\langle \alpha,\gamma\rangle$ where $\alpha$ and $\gamma$ are loops homotopic (in $H$) to the meridian and the equator of the former torus. Let $\beta$ be a parametrization of $\partial H$, with the induced orientation from $H$, changing orientations if necessary, we can assume that $\beta$ is homotopic to the commutator $[\alpha,\gamma]=\alpha^{-1}\gamma^{-1}\alpha\gamma$.

Let $p=\dfrac{z-1}{-z+2}$ and $q=\dfrac{z+1}{z+2}$ be the generators of a projective action of the free group on two generators over the circle\footnote{Identifying the circle with the projective line}. More precisely, $\langle p,q\rangle$ is isomorphic to $\Z\ast\Z$ since they can be identified with generators of deck transformations associated to a punctured torus with a cusp (see fig. \ref{f:F-H-P}). As a consequence the action is minimal by Remark~\ref{r:minimal cusps}. These maps $p,q$ are hyperbolic and $[p,q]\equiv z + 6$ is parabolic and fixes exactly the $\infty$ point of $S^1\equiv\R\cup\{\infty\}$. Let $h:\pi_1(H)\to \Diff^\omega(S^1)$ be the faithful action defined by $h([\alpha])=p$ and $h([\gamma])=q$. The suspension of the previous homomorphism defines the projective block $\FF(\Z\ast\Z,H,h)$.

\end{exmp}

\begin{exmp}\label{e:F-P}
Let $O=S^2(2,2,3,3)$ and let $\sigma$ be a loop separating the cone points in pairs so that each connected component of $\bm{O}\setminus\bm{\sigma}$ is the underlying space of a once punctured $S^2(2,3)$ orbifold. In this case each $B^\star$ are homeomorphic to a $2$-sphere with three pairwise disjoint disks removed,  i.e. a {\em pair of pants} that will be denoted by $P$. Let $\alpha,\beta$ and $\gamma$ parametrizations of the three circle boundaries of $P$, we shall assume that these parametrizations are oriented according to the orientation induced by the orientation of $P$ and $\alpha,\gamma$ are associated to cone points. See Figure~\ref{f:F-H-P} in order to check the elements of this example.

It follows that $\beta^{-1}$ is homotopic to $\alpha\ast\gamma$ and $\pi_1(P)=\Z\ast\Z=\langle[\alpha],[\gamma]\rangle$, where $[\alpha]$ and $[\gamma]$ are the homotopy classes relative to $\alpha$ and $\gamma$ respectively.

Let $f=\frac{-1}{z},g=\frac{1}{1-z}$ be the two generators of $\PSL(2,\Z)=\Z_2\ast\Z_3$, so that $f^2=\id$, $g^3=\id$ and $f\circ g=z-1$. They act analitically on the circle as M\"obius transformations. Let $h:\pi_1(P)\to \Diff^\omega(S^1)$ be the homomorphism defined by $h([\alpha])=f$ and $h([\gamma])=g$.

Let $\FF(\PSL(2,\Z),P,h)$ be the associated foliated block, its ambient space is $P\times S^1$. This foliation has exactly three transverse boundary components which are three tori. Each component is identified respectively with $\bm{\alpha}\times S^1$, $\bm{\beta}\times S^1$ and $\bm{\gamma}\times S^1$. The trace foliations on these tori will be called $\zeta_\alpha$, $\zeta_\beta$ and $\zeta_\gamma$ respectively.
\begin{itemize}
	\item $\zeta_\alpha$ is just the suspension of $f$ and it is conjugated to a linear flow associated to a $\rot(f)=1/2$ rotation (since $f^2=\id$)
                                               
	\item $\zeta_\beta$ is the suspension of $f\circ g \equiv z-1$ which is a parabolic diffeomorphism of the circle with a fixed point at $\infty$.
	
	\item $\zeta_\gamma$ is just the suspension of $g$ and it is conjugated to a linear flow associated to a $\rot(f)=1/3$ rotation (since $g^3=\id$ and $\rot(g)=1/3$).
\end{itemize}

Performing Dehn's filling on the elliptic tori we obtain the foliated projective block $\overline{\FF}(\PSL(2,\Z),P,h)$.
\end{exmp}

\begin{figure}
\centering
\includegraphics[scale=0.33,frame]{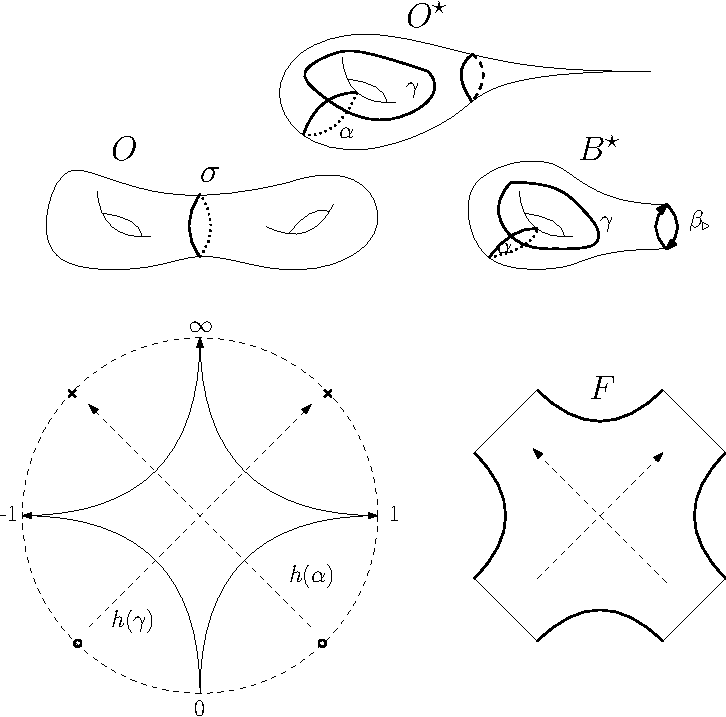}
\hspace{0.25 cm}
\includegraphics[scale=0.33,frame]{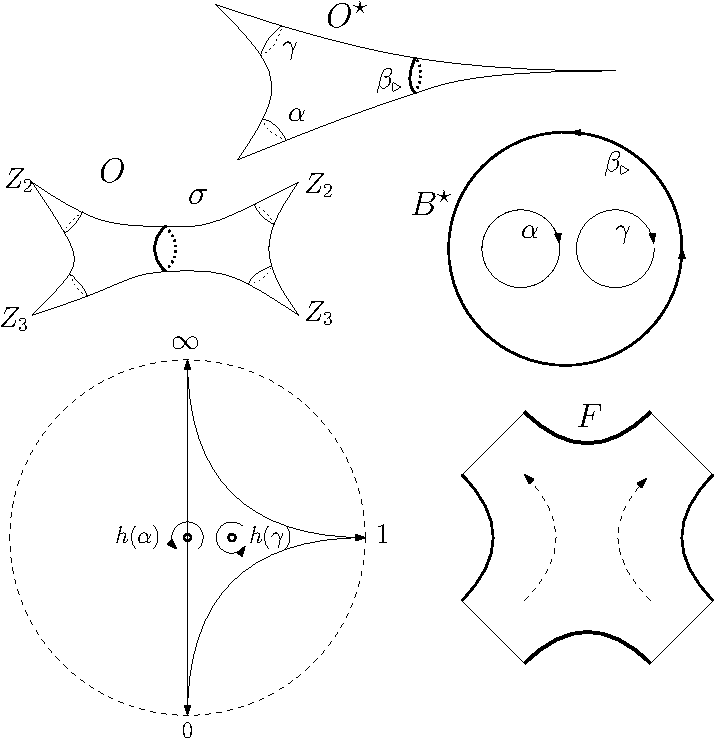}
\caption{The elements of the foliated blocks constructed in examples \ref{e:F-H} and \ref{e:F-P}.  Here, $F$ denotes a fundamental region for the action of $\pi_1(B^\star)$ on $\widetilde{B}^\star$.}\label{f:F-H-P}
\end{figure}

\section{Topology of leaves}\label{s:realization}

In order to prove Theorem~\ref{Main Theorem}, we need to codify the topology of any noncompact oriented surface in a suitable way. Our first target will be to prove the next Proposition.

\begin{prop}\label{p:surface type 0}
Every oriented noncompact surface (without boundary) can be obtained by a limit of suitable boundary unions between  countably many leaves of type $0$.
\end{prop}

Observe that a leaf of type $1$ can be obtained as a boundary union of two leaves of type $0$ identifying two pairs of boundary components. Therefore, in order to prove Proposition~\ref{p:surface type 0}, we can consider also leaves of type $1$. 

Recall also that two noncompact oriented surfaces are homeomorphic if and only if they have the same genus (possibly infinite) and their end spaces are homeomorphic via a homeomorphism preserving the planar and nonplanar ends~\cite{kerekjarto, Richards}. 
\begin{note}\label{n:step construction}
It is well known that the binary tree has a Cantor set of ends. Every closed subset $F$ of the Cantor set can be obtained as the space of ends of a connected subtree $T_F$ of the binary tree. Let $V_F$ and $E_F$ be the sets of vertices and edges of $T_F$ respectively. Let $\nu: V_F\to \{0,1\}$ be any function, that will be called a {\em vertex coloring}. Let $v\in V_{F}$, let us define $\St(v)$ as the set of edges that contain $v$ in their boundaries. Let $\deg(v)=\#\St(v)$, that will be called the {\em degree}  of the vertex $v$. Since $T_F$ is a subtree of the binary tree, it follows that $\deg(v)\leq 3$ for all $v\in V_F$. 

Without loss of generality we can assume that the root
element of the binary tree belongs to $T_F$. This root vertex will be denoted by $\bm{\mathring{v}}$. The vertices of a tree are partitioned by levels from the root element: level $0$ is just the root element, and, recursively, a vertex is of level $k$ if it is connected by an edge with some vertex of level $k-1$.

The root vertex separates the binary tree in two components, they will be called {\em positive} and {\em negative}. The {\em oriented level} of a vertex $v$ at level $k$ is defined as $k$ if $v$ lies in the positive side and $-k$ otherwise, it will be denoted by $\lvl(v)$. The oriented levels define a orientation on the binary tree just by declaring the origin of an edge $e$ as the boundary vertex with lowest oriented level, see Figure~\ref{f:binary tree}. Let $o(e)$, $t(e)$ denote the origin and target of an oriented edge.
\end{note}

\begin{defn}\label{d:step construction}
Let $T_F$ be a subtree of the binary tree. Let $K_v^0=S^2\setminus\bigsqcup_{e\in\St(v)}B_e^2$, where $\{B_e^2\}_{e\in\St(v)}$ is a collection of $\deg(v)$ pairwise disjoint open balls; and let $K_v^{1}=T^2\setminus\bigsqcup_{e\in\St(v)}B_e^2$. Let $C^e_v$ denote a boundary component of $K_v^i$ obtained by removing the disk $B^2_e$, $e\in\St(v)$. Let $\lambda^e: C^e_{o(e)}\to C^e_{t(e)}$ be orientation reversing homeomorphisms defined for each oriented edge (choose orientations on $C^e_v$ compatible with those of $K^i_v$).

Let $\nu: V_F\to \{0,1\}$ be a vertex coloring and let us consider the equivalence relation in $\bigsqcup_{v\in V_F}K_v^{\nu(v)}$ generated by the identification of points $x\in C^e_{o(e)}, y\in C^e_{t(e)}$ if and only if $\lambda^e(x)=y$. Let $S_F^\nu$ be the open connected orientable surface obtained as the quotient of $\bigsqcup_{v\in V_F}K_v^{\nu(v)}$ by this equivalence relation.
\end{defn}

It is clear that the process described above generates all the topologies of noncompact connected oriented surfaces,  see Figure~\ref{f:binary tree} for a sketch of this construction. Just for simplicity we shall consider trees without ``dead ends'', i.e., vertices where $\deg{v}=1$ are forbidden. Under this assumption the root vertex always has degree $2$. In this case the above process generates all the noncompact oriented surfaces with at least two ends. The oriented surfaces with one end are just the plane, the Loch Ness monster, and the plane with finitely many handles attached (where each number of handles defines a different topology), one-ended oriented surfaces will be treated separately.

\begin{figure}
\centering
\includegraphics[scale=0.35,frame]{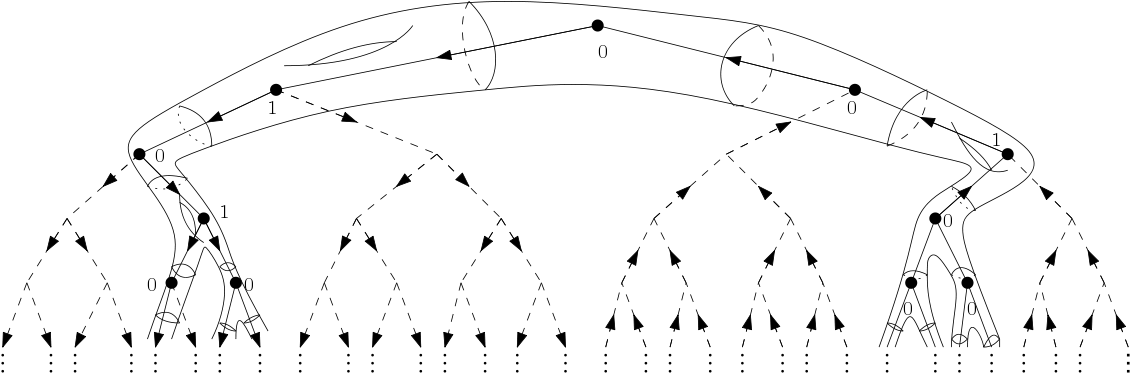}
\caption{A representation of the method of construction of $S^\nu_F$. Undotted edges represent the tree $T_F$ as a subtree of the oriented binary tree. The coloring is given by the numbers next to each vertex.}\label{f:binary tree}
\end{figure}

\subsection{Generating topology}

\begin{defn}\label{d:complete filling}
Let $N$ be an open oriented surface with boundary, assume that $\partial N$ has a noncompact boundary component, denote this component by $B$. A {\em trivial filling} of $N$ at $B$ is the surface $\widehat{N}_B$ (possibly with boundary) obtained by attaching a halfplane to $B$. This is topologically equivalent to removing the boundary component $B$ from $N$.

If $\cB$ is a subset of noncompact boundary components, then we denote by $\widehat{N}_\cB$ the surface obtained by performing trivial fillings at every $B\in\cB$.

Let $\widehat N$ denote the surface obtained by performing trivial fillings at each one of its noncompact boundary components, it will be called the {\em complete filling} of $N$ .
\end{defn}

\begin{rem}\label{r:interior=complete filling}
If $N$ has no compact boundary components, then $\widehat N$ is a surface without boundary which is clearly homeomorphic to the interior of $N$, $\Intr(N)$.
\end{rem}

\begin{rem}\label{r: plane}
A trivial filling of $N$ at some noncompact boundary component $B$ can be also obtained by the following process:
\begin{enumerate}
	\item Let $L_{0,1}$ be a leaf of type $0$ and consider $N_0=N\cup_B L_{0,1}$ the surface obtained by gluing the boundary component $B$ of $N$ 
	with any of the components of $\partial L_{0,1}$.
	\item Assume $N_{k}$ was defined. Let $L_{k,n}$, $n\in\N$, be a countable collection of distinct surfaces homeomorphic to leaves of type $0$. Let $\cB_k$ be the countably infinite set of noncompact boundary components of $N_k$  that are also boundary components of the $L_{k-1,n}$'s used in the previous step. Let $\xi_k: \cB_k\to\N$ be any bijection. Let $N_{k+1}$ be the surface obtained by gluing each boundary component of $B\in\cB_k$ with a boundary component of $L_{k,\xi_k(B)}$.
\end{enumerate}
The trivial filling $\widehat{N}_B$ is clearly homeomorphic to the surface obtained as the direct limit of the above construction. As a consequence, the complete filling of a manifold can be achieved by gluing countably many surfaces homeomorphic to leaves of type $0$ as indicated in Remark~\ref{r: plane}.
\end{rem}

\begin{note}\label{n:end_space}
Let $N$ be a noncompact oriented surface with (or without) boundary. Set $\cE(N)$ be the space of ends of $N$ with its usual topology and let $\cB(N)$ be the set of ends that are defined by noncompact boundary components (each noncompact boundary component defines two ends). An end of $N$ is called {\em planar} if there exists a neighborhood of that end without handles, otherwise it is called {\em nonplanar}.

Let $\cE'(N)$ and $\cB'(N)$ be the set of planar ends in $\cE(N)$ and $\cB(N)$ respectively. Analogously, define $\cE''(N)$ and $\cB''(N)$ be the nonplanar ends of $N$.
\end{note}

\subsection{One-ended oriented surfaces}\label{ss:one end surfaces}

Let us prove Proposition~\ref{p:surface type 0} for one-ended surfaces. The simplest case is the plane which was given  by means of Remark~\ref{r: plane} as the complete filling of a leaf of type $0$. 

In order to construct a punctured torus, i.e. a plane with a handle attached, let $L^0$ be a leaf of type $0$ and $L^1$ a leaf of type $1$. Let $B_{0}, B'_{0}$ be two arbitrary boundary components of $L^0$ and let $B_1^+$, $B_1^-$ be two boundary components of $L^1$ 
such that $B_1^+$ is outer and $B_1^-$ is inner (recall  Definition~\ref{d:out in} for the notion of outer and inner). Let $H^0 = L^0\sqcup L^1/(B_0\sim B_1^+; B'_0\sim B_1^-)$ be the surface obtained by gluing
$L^0$ with $L^1$ identifying the chosen pairs of boundary components (reversing orientations, in order to preserve the orientability of the surface), see Figure~\ref{f:H1}.

It is clear that $H^0$ has a Cantor set of ends, which are planar. The set of ends induced by  the noncompact boundary components is dense in the space of ends of $H^0$. In addition, the complete filling (or, equivalently, the interior) of $H^0$  is homeomorphic to a once punctured torus.

\begin{figure}
\centering
\includegraphics[scale=0.35,frame]{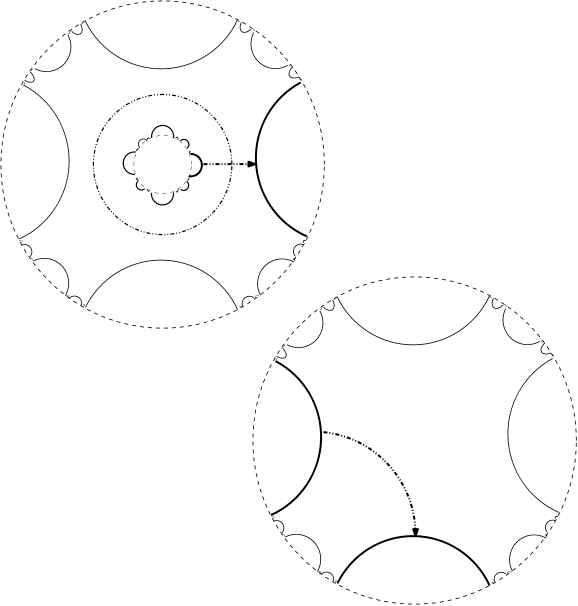}
\includegraphics[scale=0.38125,frame]{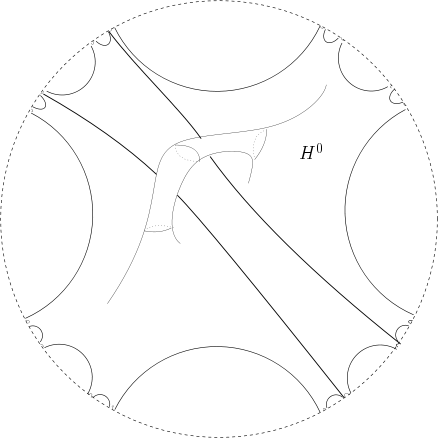}
\caption{Representation of the surface $H^0$. On the left the leaves of type $0$ and $1$ used in its construction. Bold lines are the identified boundary components.}\label{f:H1}
\end{figure}

Thus, the topology of a plane with a handle attached is realized by $\widehat{H^0}$. Let $\{M_i\}_{i=1}^k$, $k\in\N\cup\{\infty\}$, be a family of surfaces homeomorphic to $H^0$. Let $H^0_k$ denote the surface (with boundary) obtained by identifying one arbitrary boundary component of $M_i$ with another of $M_{i+1}$ for $1\leq i< k$. It follows that the plane with $k$ handles attached is just $\widehat{H^0_k}$, the case $k=\infty$ is just the Loch Ness Monster. Since a complete filling can be obtained by boundary unions with leaves of type $0$, the above construction shows that Proposition~\ref{p:surface type 0} works for one-ended surfaces.

\subsection{Oriented surfaces with more than one end}

Let us consider now a noncompact oriented surface $S$ with more than one end and let $F$ be the space of ends of $S$. Let $T_F$ be a subtree of the binary tree with no dead ends, which contains the root vertex $\bm{\mathring{v}}$ and let $\nu: V_F\to\{0,1\}$ be a coloring such that $S$ is homeomorphic to the surface $S_F^\nu$ defined in Definition~\ref{d:step construction}. 

Let $H^1$ be the surface obtained by gluing two boundary components of two leaves of type $1$ in the following way: on each leaf we choose two boundary components, inner and outer, then identify (reversing orientation) the chosen inner (resp. outer) components. 

The boundary components of $H^1$ can still be classified as outer and inner, see Figure~\ref{f:Co}.

\begin{defn}\label{d:noncompact tiles}
Let $P^i$ denote a surface homeomorphic to any leaf of type $i$, observe that $H^i$ is homeomorphic to the connected sum of $P^i$ with a torus. These manifolds with boundary will be our basic pieces to recover the topology of $S$, these manifolds $P^0,P^1,H^0,H^1$ will be called {\em noncompact tiles} in analogy with a tiling construction. Set $\cT=\{P^0,P^1,H^0,H^1\}$.
\end{defn}

\begin{figure}[h]
\centering
\includegraphics[scale=0.4,frame]{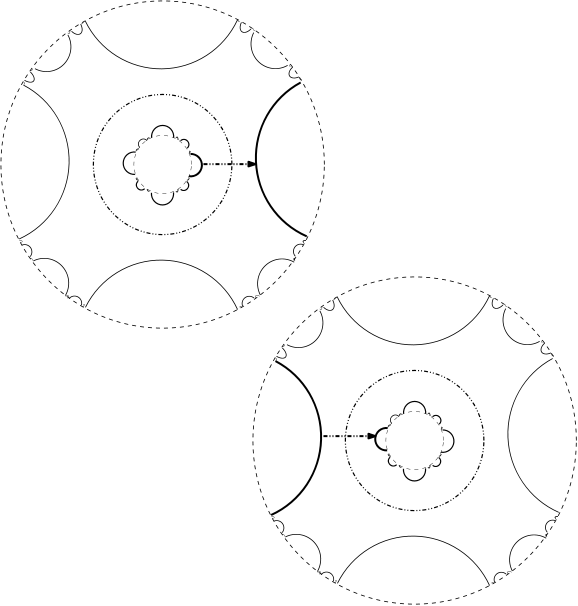}
\includegraphics[scale=0.4725,frame]{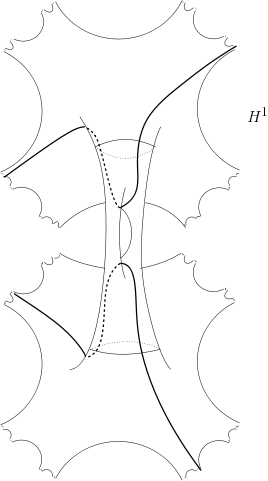}
\caption{The surface $H^1$. Outer (resp. inner) boundary components of $H^1$ are the union of outer (resp. inner) boundary components of the former two leaves of type $1$. Bold lines represent the identified boundary components.}\label{f:Co}
\end{figure}

\begin{defn}\label{d:realization W_F}
Given a tree $T_F$ as above and a coloring $\nu$, for each $v\in V_F$ we  shall choose a noncompact tile $N_v\in \cT$ as follows. 

\begin{itemize}
\item If $v = \bm{\mathring{v}}$ and $\nu(\bm{\mathring{v}})=0$, then $N_{\bm{\mathring{v}}}=P^1$. 

\item If $v = \bm{\mathring{v}}$ and $\nu(\bm{\mathring{v}})=1$, then  $N_{\bm{\mathring{v}}}=H^1$.

\item If $v\neq\bm{\mathring{v}}$, $\deg(v)=2$ and $\nu(v)=0$, then $N_v=P^0$ 

\item If $v\neq\bm{\mathring{v}}$, $\deg(v)=2$ and $\nu(v)=1$, then $N_v=H^0$.

\item If $\deg(v)=3$ and $\nu(v)=0$, then $N_v=P^1$.

\item If $\deg(v)=3$ and $\nu(v)=1$, then  $N_v=H^1$.
\end{itemize}

An oriented edge $e\in E_F$ will be also used to denote a reversing orientation homeomorphism from a boundary component $B^o_e$ of $N_{o(e)}$ to a boundary component $B^t_{e}$ of $N_{t(e)}$.

For edges $e$, $e'$, so that $t(e)=o(e')$, it will be assumed that $B^t_e$ and $B^o_{e'}$ are different boundary components of the same noncompact tile.

In the case where $\deg(v)=3$, $v$ can be the origin (resp. the target) of two distinct edges $e,e'$. It will be always assumed that $B^o_{e}$ and $B^o_{e'}$ (resp. $B^t_{e}$  and $B^t_{e'}$) are not both inner and neither both outer boundary components of $N_v$ (see Definition~\ref{d:out in}). In the specific case where $o(e)=\bm{\mathring{v}}$ (resp. $t(e)=\bm{\mathring{v}}$) is the root element, the boundary component $B^o_e$ is outer (resp. $B^t_e$ is inner).

Let $W_F^\nu$ be the surface (with boundary) given as the quotient of $\bigsqcup_{v\in V_F} N_v$ by the equivalence relation generated by the boundary gluing maps $e: B^o_e\to B^t_e$, $e\in E_F$. A sketch of this construction can be seen in Figure~\ref{f:W_F}.
\end{defn}

\begin{lemma}\label{l:prop of W_F}
The manifold $W^\nu_F$ has the following properties:
\begin{enumerate}
\item $\partial W^\nu_F$ consists of countably many noncompact boundary components.

\item The space of ends $\cE(W^\nu_F)$ is a Cantor set.

\item The genus of $W^\nu_F$ is the same as the genus of $S^\nu_F$.

\item There exists a canonical embedding $\iota:\cE(S^\nu_F)\hookrightarrow\cE(W^\nu_F)\setminus \cB(W^\nu_F)$.

\item $\cE''(W^\nu_F)=\iota(\cE''(S^\nu_F))$.

\item $\cB(W^\nu_F)$ is dense in $\cE(W^\nu_F)$.
\end{enumerate}
\end{lemma}
\begin{proof}
The unique nontrivial property is property $4$ and $5$ as the {\em canonical embedding} must be defined. Observe that $T_F$ can be embedded in $S^\nu_F$ (resp. $W^\nu_F$). Such an embedding is called {\em canonical} if each vertex $v$ is mapped into $K^i_v$ (resp. $N_v\in\cT$). These canonical embeddings induce natural mappings $\phi_1:\cE(T_F)\to\cE(S^\nu_F)$ and $\phi_2:\cE(T_F)\to\cE(W^\nu_F)\setminus\cB(W^\nu_F)$ where $\phi_1$ is clearly a homeomorphism. Define $\iota=\phi_2\circ\phi_1^{-1}$. This map must be injective since, by construction, each bifurcation of $T_F$ produces a topological bifurcation in  $W^\nu_F$. As in the construction of $S^\nu_F$, whenever $\nu(v)=1$ we include a handle in the construction, thus nonplanar ends of both $S^\nu_F$ and $W^\nu_F$ correspond exactly with the ends of $T_F$ accumulated by the color $1$ and therefore are canonically identified. Observe also that every end in $\cE(W^\nu_F)\setminus \iota(\cE(S^\nu_F))$ must be planar since all the noncompact tiles in $\cT$ have planar ends.
\end{proof}

\begin{figure}[h]
\centering
\includegraphics[scale=0.6,frame]{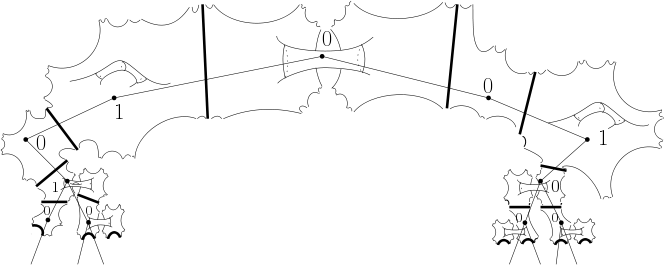}
\caption{The surface $W^\nu_F$. Bold lines represent boundary components identified by edges. The numbers represent a coloring of $V_F$, similar to those used in figure~\ref{f:binary tree}. Observe that a handle appear for each vertex whose color value is equal to $1$.}\label{f:W_F}
\end{figure}

\begin{prop}\label{p:W=S}
The interior of $W^\nu_F$ is homeomorphic to $S^{\nu}_F$.
\end{prop}
\begin{proof}
The interior of $W^\nu_F$ is obtained by removing its boundary components. It turns out that $\Intr(W^\nu_F)$ is obtained by suitable boundary gluings between noncompact tiles were all their noncompact boundary components are removed except for those components that are attached.

Let $N^{\circ}_v$ be the manifold obtained by removing the noncompact boundary components of $N_v$ except those (at most three) that are attached to other noncompact tiles in the construction of $W^\nu_F$. All these manifolds have the following property:

\textbf{Property A:} If $\lvl(v)>0$ (resp. $\lvl(v)<0$), then for any compact set $K\subset N_v$ and any point $x\in N_v^\circ\setminus K$ there exists a path $\gamma_v:[0,1]\to N_v^\circ\setminus K$ such that $\gamma_v(0)=x$ and $\gamma_v(1)\in B^o_{e}$ (resp. $\gamma(1)\in B^t_{e}$) for some edge with $o(e)=v$ (resp. $t(e)=v$). If $\lvl(v)=0$, i.e. $v=\mathring{\bm{v}}$, then any point in $x\in N_v^\circ\setminus K$ can be conencted by a path with some of its two boundary components.

It follows that $\cE(\Intr(W^\nu_F))=\iota(\cE(S^{\nu}_F))$, since Property A allows to construct (inductively) a semiinfinite proper path beginning at any point in the complement of a compact set and such that its end do not lie in $\cE(N^\circ_v)$ for all $v\in V_F$, hence it defines an end in $\iota(\cE(S^{\nu}_F))$.

By the properties listed in \ref{l:prop of W_F}, it follows that $\Intr(W^\nu_F)$ and $S^\nu_F$ have the same space of ends and the homeomorphism $\iota:\cE(S^\nu_F)\to\cE(\Intr(W^\nu_F))$ preserves planar and nonplanar ends. Since both manifolds have the same genus by construction and are both oriented, it follows from the classification Theorem of noncompact surfaces \cite{kerekjarto, Richards} that $\Intr(W^\nu_F)$ is homeomorphic to $S^\nu_F$.
\end{proof}

\begin{proof}[Proof of Proposition~\ref{p:surface type 0}]
Recall that $\Intr(W^\nu_F)$ is also homeomorphic to the complete filling (see Definition~\ref{d:complete filling}) of $W^\nu_F$ and this can be also obtained by suitable gluings between noncompact tiles homeomorphic to $P^0$ (see Remark~\ref{r: plane}). Therefore Proposition~\ref{p:surface type 0} follows from Proposition~\ref{p:W=S} and Subsection~\ref{ss:one end surfaces}.
\end{proof}

\section{Generic transverse gluing}\label{s:transverse gluing}

Let $\FF^\triangleright$ and $\FF^\triangleleft$ be foliated projective blocks (as defined in Section~\ref{s:projective blocks}) constructed from a typical orbifold $O$, a separating loop $\sg$ and suitable representations $h^\star:\pi_1(O^\star)\to\PSL(2,\R)$. The parabolic tori of $\FF^\triangleright$ and $\FF^\triangleleft$ can be transversely glued by an orbit preserving homeomorphism since the trace foliations of the parabolic tori are conjugated to each other. This is a direct consequence of Proposition~\ref{p:always parabolic}.

 Assume that the typical orbifold is a closed surface. Any attaching map between the parabolic tori which is fiber preserving (but that does not necessarily preserve the trace foliations), can be written in the form $(x,y)\mapsto (x,\psi(x)(y))$, where $\psi: \partial B^\triangleright\to \Homeo_+(S^1)$. The isotopy class of this attaching map depends just in the choice of orientations and the degree of the map $x\to \psi(x)^{-1}(y_0)$ for any distinguished $y_0\in S^1$, this degree is denoted by $m(\psi)$. If $M(O,\sg,\psi)$ denotes the Seifert fibration constructed in this way, its Euler number is $m(\psi)$.

 {\bf Definition 3.1.} Given $\psi: \partial B^\triangleright\to \Homeo_+(S^1)$, it is said that $\psi$ (and hence also the homeomorphism $(x,y)\mapsto$ $(x,\psi(x)(y))$ and the degree $m(\psi)$) is {\em admissible} if the trace foliations on the parabolic tori are preserved, this specific case will be noted by $M(O,\sg)$.

\begin{rem}\label{r:ambient manifold}
Let $\eu(h_\triangleright,s_{h_\triangleright})$ and $\eu(h_\triangleleft,s_{h_\triangleleft})$ denote the (relative) Euler numbers defined in Remark~\ref{r:relative euler class}. If $\psi$ preserves the trace foliations, then it should satisfy that $$|m(\psi)|=|\left(\eu(h_\triangleright,s_{h_\triangleright})+\eu(h_\triangleleft,s_{h_\triangleleft})\right)|$$ which is the Euler number of the Seifert fibration (see \cite[Lemma 3.1]{Ghys2}).

More precisely, if $O$ is a closed hyperbolic surface and $\sg$ is a separating loop, then $$|\eu(M(O,\sg))|=|\chi(B^\triangleright)-\chi(B^\triangleleft)|\;.$$ Note that this is the unique possibility since $|\eu(M(O,\sg)|$ cannot be maximal by Ghys' Theorem \cite[T\'eor\'eme 3]{Ghys2}. 

This formula also works in the case where $O$ is any typical orbifold. It is not difficult to give the symbolic representation of $M(O,\sg)$ as a Seifert fibration from the previous data, it just remains to give structure constants associated to each cone point. These constants are the rotation numbers of the elliptic generators chosen in the presentations of $G^\triangleright$ and $G^\triangleleft$,  thus the topology of $M(O,\sigma)$ may depend on the choice of $h^\star$.
\end{rem}

\begin{rem}\label{r:Selberg}
Selberg's Lemma~\cite{Selberg} implies that any hyperbolic $2$-dimensional orbifold is finitely covered by a closed surface of genus $g\geq 2$. It follows that $M(O,\sg)$, as a fibration, is finitely covered by a closed $3$-manifold which is a circle fiber bundle over a hyperbolic surface.
\end{rem}

If an admissible homeomorphism $\psi$ is $C^r$, then we obtain a $C^r$ foliation on the closed $3$-manifold $M(O,\sg)$. This foliation will be minimal since every foliated block is minimal and its ambient manifold is Seifert where its associated base orbifold is the original orbifold $O$.

The simplest case of an admissible gluing map occurs when $\psi(x)(y)=n x + y$ (mod $2\pi$), where $n$ is the Euler class of $M(O,\sg)$. In this case the resulting foliation is analytic (if $n=0$, the attaching map is just the identity).

Let $T^\triangleright$ and $T^\triangleleft$ be complete transverse circles (fibers), in the parabolic boundary tori of $\FF^\triangleright$ and $\FF^\triangleleft$, we shall identify $T^\star$ with the projective line $\R\cup\{\infty\}$. The parabolic boundary tori can be identified with $S^1\times T^\triangleright$ and $S^1\times T^\triangleleft$ respectively, let $\zeta_\triangleright$ and $\zeta_{\triangleleft}$ be their respective trace foliations and let $\ell=h(\beta_\triangleright)$ and $\bar{\ell}=h(\beta_\triangleleft)^{-1}$, which are projective diffeomorphisms defined over $T^\triangleright$ and $T^\triangleleft$ respectively.

Let $$\Psi: (S^1\times T^\triangleright,\zeta_\triangleright)\to (S^1\times T^\triangleleft,\zeta_{\triangleleft})$$ be an admissible homeomorphism of the form  $(x,y)\to$ $(x,\psi(x)(y))$, where $\psi:S^1\to\Homeo_+(T^\triangleright,T^\triangleleft)$. The map $\psi(0):T^\triangleright\to T^\triangleleft$ determines univocally $\Psi$. More precisely, let $\Phi_\triangleright(x)$ and $\Phi_{\triangleleft}(x)$ be the flows whose orbits define, respectively, $\zeta_\triangleright$ and $\zeta_{\triangleleft}$, and such that their tangent vector fields projects to $\partial/\partial\theta$ in $S^1\equiv \R/\Z$ (hence $\Phi_\triangleright(1)(y) = \ell(y)$ and $\Phi_\triangleleft(1)(y) = \bar{\ell}(y)$), it follows that $\psi(x)(y)=\Phi_{\triangleleft}(x)\circ\psi(0)\circ \Phi_\triangleright(-x)(y)$.

There are not many possibilities for such a gluing map $\Psi$ when $\psi(0)$ is $C^1$. This is consequence of the so called Kopell Lemma that will be recalled here.

\begin{lemma}[Kopell Lemma \cite{Navas}]
Let $f$ and $g$ be a pair of commuting diffeomorphisms of $[0,\infty)$ so that both $f$ and $g$ fix $0$, $f$ is of class $C^2$ and is topologically contracting at $0$ and $g$ is of class $C^1$. Then, if $g$ fixes another point, then $g$ is the identity map.
\end{lemma}

\begin{cor}\label{c:corollary Kopell}
Let $f$ be a parabolic (resp. hyperbolic) circle projective diffeomorphism. The centralizer of $f$ in $\Diff^1_+(S^1)$ is exactly the set of parabolic  (resp. hyperbolic) projective diffeomorphisms that fix the same points as $f$.
\end{cor}

Recall that the trace foliations of the parabolic boundary tori are also suspensions of the parabolic projective diffeomorphisms $\ell$ and $\overline{\ell}$ (in particular $C^2$), both in the same projective conjugacy class by means of Proposition~\ref{p:always parabolic}. Let $\varsigma$ a projective diffeomorphism such that $\varsigma\circ \overline{\ell}\circ \varsigma^{-1}=\ell$. Since $\Psi$ is leaf preserving it follows that $\overline{\ell}\circ \psi(0)=\psi(0)\circ\ell$ and hence $\varsigma\circ\psi(0)$ must commute with $\ell$. It follows, by means of the previous Corollary~\ref{c:corollary Kopell}, that $\varsigma\circ\psi(0)$ must be a projective map, hence $\psi(0)$ is also projective and the resulting foliation is still projective.

One can think that there are still lots of possibilities for gluing maps since the family of parabolic elements with a common fixed point is in correspondence with the translations of the real line (via stereographic projection). Unfortunately, for each possible such gluing the topology of leaves is still very restrictive as it is stated in the next Proposition (see also \cite[Proposition 1]{ADMV2}).

\begin{prop}\label{p:projective topology}
Let $\Sigma$ be a closed surface and let $h: \pi_1(\Sigma)\to\PSL(2,\R)$ be a homomorphism. Let let $\FF_h$ be the corresponding suspension foliation, assume that it is minimal. If some leaf of $\FF_h$ is of finite type, then every leaf has also finite type and is homeomorphic to a plane or a cylinder.
\end{prop}
\begin{proof}
If $h$ is not injective, then every leaf has a loop which is nontrivial in homotopy but trivial in holonomy (corresponding to elements in $\Ker h$). Since the action is minimal, it follows, by \cite[Theorem 2]{ADMV}, that every leaf has infinite type. Thus, only suspensions of minimal faithful actions can have leaves with finite type. Since the action is projective, it is still true that stabilizers must be solvable groups. Every leaf is a noncompact surface by minimality, thus their fundamental groups are free. Let $x\in S^1$ and let be $L_x$ be the leaf of $\FF_h$ meeting $x$. Since $h$ is injective, this stabilizer is isomorphic with $h(\Stab_{h}(x))$  which is the stabilizer of $x$ by the action of the projective group $h(\pi_1(\Sigma))$, thus it is solvable. It follows, from Lemma~\ref{l:leaf topology suspension}, that the fundamental group of every leaf is isomorphic to a solvable free group and therefore it must be trivial or infinite cyclic. The trivial case corresponds with the plane and the cyclic one with the cylinder.
\end{proof}

\subsection{From torus to interval}

There are, however, many possibilities to realize a $C^0$ gluing map between the parabolic tori preserving the orbits of their trace foliations.

\begin{assumption}
We shall make some assumptions:
\begin{enumerate}
\item $\infty\in S^1$ is the fixed point of $\ell$. 

\item $\bar{\ell}=\ell$.
\end{enumerate}

Recall that both $\ell$ and $\bar{\ell}$ belong to the same projective conjugacy class. Thus, modulo a projective conjugation in the group $G^\triangleright$ or $G^\triangleleft$ (see Notation~\ref{n:block data}), our assumptions cover all the difficulties of the general case but simplifies the notations.

Although $\ell=\bar{\ell}$ are, formally, the same projective maps, they are considered as functions in different transverse circles, so the notation $\bar{\ell}$ will be still used in order to avoid confusion.
\end{assumption}

Let $x_0,y_0\in S^1\setminus\{\infty\}$ be arbitrary points. Clearly, $[x_0,\ell(x_0)]$ and $[y_0,\bar{\ell}(y_0)]$ are 
fundamental domains for the action of $\ell$ and $\bar{\ell}$ in $T^\triangleright\setminus\{\infty\}$ and $T^\triangleleft\setminus\{\infty\}$ respectively. Let $r: [x_0,\ell(x_0)]\to [y_0,\bar{\ell}(y_0)]$ be an orientation preserving homeomorphism, then the map
$$
c(x)=c_{r,x_0,y_0}(x)=
\begin{cases}
\infty&\ \text{if} \  x=\infty\\

\left(\bar{\ell}\right)^k\circ r \circ (\ell)^{-k}(x) \ \text{if} \  x\in[\ell^k(x_0),\ell^{k+1}(x_0)[&
\end{cases}
$$
defines a homeomorphism of the circle which maps the orbits of $\ell$ to orbits of $\bar{\ell}$. This is equivalent to $\bar{\ell}\circ c=c\circ\ell$. Moreover, every orbit preserving homeomorphism from $(T^\triangleright,\ell)$ to $(T^\triangleleft,\bar{\ell})$  is determined univocally by its restriction to a fundamental domain of $\ell$ as above. 

Of course, the circle map $c\equiv c_{r,x_0,y_0}$ can be used to define a torus transformation $$\Psi_{c}:(S^1\times T^\triangleright,\zeta_\triangleright)\to (S^1\times T^\triangleleft,\zeta_{\triangleleft})\;.$$ Therefore there is a bijective correspondence between admissible homeomorphisms $\Psi$ between parabolic boundary tori of foliated blocks and $\Homeo_+([x_0,\ell(x_0)],[y_0,\bar{\ell}(y_0)])$.

\begin{notation}\label{n:notation_c}
Fix $x_0$ and $y_0$ as $0$ in $\R\cup\{\infty\}$. Let $\cZ^0_+(\ell;0)$ denote the group of orientation preserving circle homeomorphisms that commute with $\ell$ and fix $0$. And let $\cZ^{1\ast}_+(\ell;0)$ be the family of circle homeomorphism $c\in\cZ^0_+(\ell;0)$ such that $c_{|[0,\ell(0)]}\in \Diff^1_+([0,\ell(0)])$ and it is tangent to the identity at its extreme points.

For each $c\in\cZ^{1\ast}_+(\ell;0)$, the corresponding torus homeomorphism between the parabolic boundary tori will be denoted by $\Psi_c$. Remark that every map $c\in\cZ^{1\ast}_{+}(\ell;0)$ will fix at least $\infty$ and the $\ell$-orbit of $0$.

Observe also that for any $x\in S^1\setminus\{\infty\}$ there exists a unique $k_x\in\Z$ so that $\ell^{k_x}(x)\in[0,\ell(0)[$, we shall denote the point $\ell^{k_x}(x)$ by $\overline{x}$. For future reference, we shall also define $k_\infty=0$ and $\overline{\infty}=\infty$. 
\end{notation}

\begin{rem}
Remark that there is a natural identification of $\cZ^{1\ast}_+(\ell;0)$ with a closed subset of $\Diff^1_+([0,\ell(0)])$, we shall endow $\cZ^{1\ast}_+(\ell;0)$ with a topology induced by the $C^1$-topology in $\Diff^1_+([0,\ell(0)])$ via this identification.
\end{rem}

\begin{rem}\label{r:derivative}
Observe that we are dealing with circle homeomorphisms as homeomorphisms in the projective line, this can cause confusion when derivatives are computed. The derivatives for circle maps are computed via a lift of $S^1=\R/\Z$ to $\R$, for any differentiable $f:\R/\Z\to\R/\Z$ its (usual) derivative will be denoted by $df/d\theta$ or simply by $d_\theta f$.

When $f$ is analitically defined in $\R\cup\{\infty\}$ and preserving $\infty$ (as it is done along this section) the notation $f'(x)$ for $x\in \R$, denotes its derivative as a function of the real line. The relation between both derivatives is the following: if $[0]\in\R/\Z$  represents the $\infty$ point and $p:\R/\Z\setminus\{[0]\}\to\R$ is the stereographic projection, then $d_\theta f(z)=(p^{-1}\circ f\circ p)'(z)$ for $z\in \R/\Z\setminus\{[0]\}$. 

Observe also that, when $0$ is a fixed point of $f:\R\cup\{\infty\}\to\R\cup\{\infty\}$, we have $d_\theta f(p^{-1}(0))=f'(0)$ (the stereographic projection from the north pole has no distorsion at the south pole, which is $p^{-1}(0)$).
\end{rem}

\begin{lemma}\label{l:regularity}
Every map in the group $\cZ^{1\ast}_+(\ell;0)$ is differentiable and its derivative (as circle maps) is uniformly bounded (in particular these are Lipschitz maps).
\end{lemma}
\begin{proof}
Let $c\in \cZ^{1\ast}_+(\ell;0)$. Since $c_{|[0,\ell(0)]}$ has derivative $1$ at $0$ and $\ell(0)$, and $\ell$ is parabolic fixing $\infty$ (so it is a translation in the projective line) it follows that $c$ is differentiable in $S^1\setminus\{\infty\}$.

Let us show that $c$ is differentiable at $\infty$. Observe that $c(\ell^n(0))=\ell^n(0)=n\cdot\ell(0)$ for all $n\in\Z$ by construction. The inversion $1/x$ defines the change of coordinates of the stereographic projection from the north to the south pole and maps $\infty$ to $0$.

Thus, by Remark~\ref{r:derivative}, $d_\theta c(\infty)=\lim\limits_{x\to\pm \infty}\dfrac{x}{c(x)}$ that clearly converges to $1$ since $\dfrac{n}{n+1}\leq \dfrac{x}{c(x)}\leq\dfrac{n+1}{n}$ for $x\in [\ell^n(0),\ell^{n+1}(0)[$ and $n\to\pm\infty$ as $x\to \pm\infty$.

Recall now that $\ell$ is an analytic circle diffeomorphism, in particular $C^2$. Let $\hat{c}$ be the restriction of $\varphi$ to $[0,\ell(0)[$ and let $x\in [\ell^{n}(0),\ell^{n+1}(0)[$ for $n\geq 0$. Set $y_i=\ell^{n-i}\circ{\hat{c}}(\overline{x})$ and $x_i= \ell^{n-i}(\overline{x})$, $1\leq i\leq n$. Thus 
$$
d_\theta c(p^{-1}(x))=d_\theta (\ell^n\circ \hat{c}\circ\ell^{-n})(p^{-1}(x))=d_\theta\hat{c}(p^{-1}(\bar{x}))\prod_{i=1}^n \dfrac{d_\theta\ell(p^{-1}(y_i))}{d_\theta\ell(p^{-1}(x_i))}= d_\theta\hat{c}(p^{-1}(\bar{x}))\dfrac{d_\theta\ell^n(p^{-1}(y_0))}{d_\theta\ell^n(p^{-1}(x_0))}\;.
$$

Recall that here the derivatives are computed as circle maps, not as transformations of the (projective) real line and $p$ is the stereographic projection (see Remark~\ref{r:derivative}). By the same argument as in \cite[Lemma 8.1.3]{Candel-Conlon-I-2000} there exists $\Theta>0$, defined as the quotient of the maximum of $d^2_\theta\ell$ and the minimum of $d_\theta\ell$ in $p^{-1}([0,\ell(0)])$ such that
$$
\dfrac{d_\theta\ell^n(\hat{c}(p^{-1}(\bar{x})))}{d_\theta\ell^n(p^{-1}(\bar{x}))}\leq \exp\left(\Theta\sum_{i=1}^n \left|p^{-1}(y_i)-p^{-1}(x_i)\right|\right)
$$ where $|a-b|$ denotes the arc distance between two points $a,b\in S^1$. Since each pair $x_i,y_i$ lies in $\ell^{n-i}([0,\ell(0)])$, and these domains are pairwise disjoint, it follows that $\sum_{i=1}^n|p^{-1}(y_i)-p^{-1}(x_i)|$ is uniformly bounded by the length of the circle.

A similar argument works for $n\leq 0$ (just by considering $\ell^{-1}$ instead of $\ell$). Therefore, the derivative of $c$ (as a circle diffeomorphism) is uniformly bounded and thus $c$ is Lipschitz as claimed.
\end{proof}

\subsection{Generic gluings}

The foliation obtained from $\FF^\triangleright$ and $\FF^\triangleleft$ by identifying the transverse parabolic tori via a homeomorphism $\Psi_c$ (defined in Notation 3.8) from an element $c\in\cZ^{0}_+(\ell;0)$ can also be considered as a suspension of a group action plus Dehn fillings. 

This suspension can be described as an amalgamated product of $G^\triangleright$ and $G^\triangleleft$  over $\Z$, this cyclic subgroup is generated by $h_\triangleright(\beta_\triangleright)=\ell=\overline{\ell}=h_\triangleleft(\beta_\triangleleft)^{-1}$. This induces a natural homomorphism,
$h: \pi_1(B^\triangleright)\ast_\Z \pi_1(B^\triangleleft)\to G^\triangleright\ast_{\langle\ell\rangle}G^\triangleleft$. Let $\Gamma^{\Join}_\ell=G^\triangleright\ast_{\langle\ell\rangle}G^\triangleleft$ and let us consider the homomorphism $\rho_{c,d}: \Gamma^{\Join}_\ell\to\Homeo_+(S^1)\,;$ depending on $c,d\in\cZ^0_+(\ell;0)$ and defined as

\[
\rho_{c,d}(g)=
\begin{cases}
c\circ g\circ c^{-1} &\ \text{if} \ g\in G^\triangleright\\

d^{-1}\circ g\circ d &\ \text{if} \ g\in G^\triangleleft\;.
\end{cases}
\]

The fact that $c$ and $d$ commute with $\ell$ guarantees that the above action is well defined in the amalgamated product.
The above is equivalent to consider the gluing map associated to the homeomorphism $\Psi_{d\circ c}$ between the parabolic boundary tori of these foliated blocks. 

\begin{defn}\label{d:F_c-d}
The homomorphism $\rho_{c,d}$ defines the left action $\rho_{c,d}(g)\cdot x:=\rho_{c,d}(g)^{-1}(x)$ for every $g\in \Gamma_\ell^{\Join}$. Let $B^{\Join}$ be the surface with boundary obtained from $B^\triangleright$ and $B^\triangleleft$ identifying the loops $\beta_\triangleright$ and $\beta_\triangleleft^{-1}$, preserving these orientations.

The foliation defined by the suspension of the homomorphism $\rho_{c,d}\circ h$ over $B^{\Join}$ will be denoted by $\FF_{c,d}$. Let $\FF^{\Join}_{c,d}$ be the foliation without boundary obtained by performing Dehn's fillings on all the elliptic boundary tori of $\FF_{c,d}$\footnote{Observe that $\FF_{c,d}$ have elliptic boundary components if and only if the singular locus of the tyical orbifold $O$ is nontrivial.}.
\end{defn}

\begin{rem}[Representation in normal form]\label{r:amalgalm}
Let $G_1\ast_A G_2$ be an amalgamated product between two groups $G_1,G_2$ over a common subgroup $A$. Let $S_1\subset G_1$ and $S_2\subset G_2$ be a full set of representatives of left cosets, i.e. $\# S_i\cap gA = 1$ for all $g\in G_i$, $i=1,2$. Assume that $1\in S_i$ (representing the coset $A$). Let ${\bf i}=(i_1,\dots, i_n)$ be an {\em alternating} multi-index, i.e. $i_j\in\{1,2\}$ and $i_{j+1}\neq i_j$ for all $j$. A {\em reduced word}\footnote{It is standard to choose these reduced words by representatives of right cosets instead of left ones. Observe that a reduced word by left cosets is transformed in a reduced word by right cosets via inversion.} of type $\bf{i}$ is any family $(s_1,\dots, s_n;a)$ such that $a\in A$, $s_j\in S_{i_j}$ and $s_j\neq 1$ for all $j$.

It is well known, see e.g. \cite[Theorem 1]{Serre}, that every $g\in G_1\ast_A G_2$ admits a unique reduced word such that $g= s_1\cdots s_n\cdot a$. Moreover, if ${\bf j}=(j_1,\dots, j_m)$ is an alternating multindex such that $g=t_1\cdots t_m$  and $t_{i}\in G_{j_i}\setminus A$ for $1\leq i\leq m$, then $\bf{i}=\bf{j}$ and there exists $a_i\in A$, for $0\leq i\leq n$ with $a_0=1$ such that $s_i =a_{i-1}^{-1} t_i a_i$ for all $1\leq i\leq n$. This can be easily proved recursively and is also a consequence of \cite[Remark to Theorem 1]{Serre}. Any of the representations as above will be called {\em normal representations} of $g$. Let $|g|$ denote the length of the multi-index associated to any normal representation of $g$.  In this work we have $G_1=G^\triangleright$, $G_2=G^\triangleleft$ and $A=\langle \ell\rangle$.
\end{rem}

\begin{notation}\label{n:notation 1}
In the following for $w\in \Gamma^{\Join}_\ell$ we set $w_{c,d}=\rho_{c,d}(w)^{-1}$. 

As we showed in Subsection 3.1, $\cZ^0_+(\ell;0)$ is in natural correspondence with $\Homeo_+[0,\ell(0)]$, this bijection is also a homeomorphism relative to the $C^0$-topologies on these sets.
\end{notation}

\begin{lemma}\label{l:perturbation}
Let $f\in\Homeo_+(S^1)$ and let $\ell: S^1\to S^1$ be a circle parabolic homeomorphism fixing $\infty\in S^1$. Let $x\in S^1\setminus\{\infty\}$ and set $y=f(x)$, suppose that $0\neq\overline{x}\neq\overline{y}\neq 0$. For any neighborhood $V$ of $\overline{x}$, there exists $\varphi\in\cZ^{1\ast}_+(\ell;0)$, arbitrarily close to the identity in the $C^0$-topology such that its restriction $\varphi_{|[0,\ell(0)]}$ is supported in $V$, and $\varphi^{-1}\circ f\circ\varphi(x)\neq y$.
\end{lemma}
\begin{proof}
Take $\varepsilon>0$ so that $]x-\varepsilon,x+\varepsilon[$ is  included in $S^1\setminus\{\ell^k(0)\mid\ k\in\Z\}$ and is disjoint from $f(]x-\varepsilon,x+\varepsilon[)$. It exists since $f$ is continuous, $\overline{x}\neq 0$ and $y\neq x$.

We have that $\ell^{k_y}\circ f\circ \ell^{-k_x}(\overline{x})=\overline{y}$. Since $\overline{y}\neq 0$, there exists $0<\delta<\varepsilon$ such that 
\[
]\overline{x}-\delta,\overline{x}+\delta[\subset V\cap]0,\ell(0)[
\] 
and is disjoint with 
\[
\ell^{k_y}\circ f\circ \ell^{-k_x}(]\overline{x}-\delta,\overline{x}+\delta[)
\] that follows by continuity and the fact that $\overline{x}\neq\overline{y}$.

Let $b: [0,\ell(0)]\to [0,1]$ be a $C^\infty$ nonnegative function which is zero exactly at points in $[0,\ell(0)]\setminus ]\overline{x}-\delta,\overline{x}+\delta[$ and it is $C^\infty$-close to the constant zero map. Let $b(t)\frac{\partial}{\partial t}$ be the resulting vector field in $[0,\ell(0)]$ and let $\psi_t: [0,\ell(0)]\to[0,\ell(0)]$ its corresponding flow. That is a $1$-parameter family of $C^\infty$ diffeomorphisms. Since $b$ is supported in a open subset of the interior of $[0,\ell(0)]$ it follows that $\psi_t$ is $C^\infty$ tangent to the identity in the extreme points for every $t\in\R$.

The flow $\phi_t: S^1\to S^1$ defined by $\phi_t(\infty)=\infty$ and $\phi_t(z)=\ell^{-k_z}\circ \psi_t\circ \ell^{k_z}(z)$,  for $z\in S^1\setminus\{\infty\}$, belongs to $\cZ^{1\ast}_+(\ell;0)$ for all $t\in\R$. It is also clear that, for all $t\neq 0$, $$\phi_{-t}\circ \ell^{k_y}\circ f\circ \ell^{-k_x}\circ\phi_{t}(\overline{x})\neq\overline{y}\;,$$ since, for every 
$t\neq 0$, \[\phi_{t}(\overline{x})\in]\overline{x}-\delta,\overline{x}+\delta[\setminus\{\overline{x}\}\] and thus, from $\phi_t(\overline x)\neq \overline x$, \[\ell^{k_y}\circ f\circ \ell^{-k_x}(\phi_{t}(\overline{x}))\notin ]\overline{x}-\delta,\overline{x}+\delta[\cup\{\overline{y}\}\,.\] 
Since the flow $\phi$ is supported in $]\overline{x}-\delta,\overline{x}+\delta[$ we have that $\phi_{-t}(\ell^{k_y}\circ f\circ \ell^{-k_x}\circ\phi_{t}(\overline{x}))=\ell^{k_y}\circ f\circ \ell^{-k_x}\circ\phi_{t}(\overline{x})\neq \overline{y}$. Therefore
\begin{align*}
&\phi_{-t}\circ f\circ\phi_t(x)=\phi_{-t}\circ f\circ \ell^{-k_x}\circ\phi_t(\overline{x})=\\
\ell^{-k_y}&\left(\phi_{-t}\circ \ell^{k_y}\circ f\circ \ell^{-k_x}\circ\phi_{t}(\overline{x})\right)\neq \ell^{-k_y}(\overline{y})=y\;.
\end{align*}

Therefore, choosing $\varphi=\phi_t$ (any choice of $t\neq 0$ sufficiently close to $0$ works), we obtain a homeomorphism satisfying all the required conditions.
\end{proof}

\begin{cor}\label{c:perturbation}
Let $f\in\Homeo_+(S^1)$, $x_0\in S^1\setminus\{\infty\}$ and let $\ell: S^1\to S^1$ be a circle parabolic homeomorphism fixing $\infty\in S^1$. Suppose that $\overline{x}_0\neq 0$, $f(x_0)=x_0$ and $f$ is not a germ of the identity at $x_0$. Let $V$ be an open neighborhood  of  $\bar{x}_0$. There exists $\varphi\in\cZ^{1\ast}_+(\ell;0)$, such that its restriction $\varphi_{|[0,l(0)]}$ is supported in $V$ and $\varphi^{-1}\circ f\circ\varphi(x_0)\neq x_0$.
\end{cor}
\begin{proof}
Without loss of generality, assume that there exists $f_1,f_2\in \Homeo_+(S^1)$, so that $f=f_2\circ f_1$, $x_1=f_1(x_{0})\neq \infty$ and $0\neq\overline{x}_1\neq \overline{x}_0$. Just take $f_1$ to be any homeomorphism mapping $x_0$ to an element $x_1$ with the desired properties (for instance a rotation) and set $f_2=f\circ f_1^{-1}$.

 Observe first that $f_1$, $\overline{x}_0$ and $\overline{x}_1$ satisfy the hypotheses of Lemma~\ref{l:perturbation}. Thus, we can use the same flow $\phi$ used in the proof of that Lemma such that $\phi_{|[0,\ell(0)]}$ is supported  in a neighborhood of $\bar{x}_0$ inside $V$ that does not meet $\overline{x}_1$ and $\phi_{-t}\circ f_1\circ\phi_t(x_0)= f_1\circ \phi_t(x_0)\neq x_1$ for all $t\in\R\setminus\{0\}$. 

Set $y_t = f_1\circ\phi_t(x_0)$, the set $Y=\{y_t\mid\ t\in\R\}$ is an open neighborhood of $x_1$. Since $f$ is not a germ of the identity it follows that, for a sequence of points $t_n$ converging to $0$ as $n\to\infty$, we have that $f_2(y_{t_n})\neq \phi_{t_n}(x_0)$ and therefore
\begin{align*}
\phi_{-t_n}\circ f\circ \phi_{t_n}(x_0)&=\phi_{-t_n}\circ f_2\circ f_1\circ \phi_{t_n}(x_0)=\\
\phi_{-t_n}\circ f_2(y_{t_n})&\neq \phi_{-t_n}(\phi_{t_n}(x_0))=x_0
\end{align*}
\end{proof}

\subsection{Predefined stabilizers}

\begin{defn}Let $\textbf{\texttt{a}}=\{a_n\}_{n\in\Z}$ and $\textbf{\texttt{b}}=\{b_n\}_{n\in\Z}$ be 
two increasing bisequences\footnote{i.e., a sequence indexed in $\Z$.} of points in $]0,\ell(0)[$. It is said that $\textbf{\texttt{a}}$ and $\textbf{\texttt{b}}$ are {\em synchronized} if 
\[
\lim_{n\to\infty} a_n=\lim_{n\to\infty} b_n=  \ell(0)\;, \  
\lim_{n\to -\infty}a_n=\lim_{n\to -\infty}b_n= 0\;.
\]
and
\[
\lim_{n\to\infty} \dfrac{b_n-\ell(0)}{a_n-\ell(0)}=1=\lim_{n\to -\infty}\dfrac{b_n-0}{a_n-0}\;.
\]
\end{defn}

\begin{defn}
Let $\textbf{\texttt{a}},\textbf{\texttt{b}}$ be synchronized bisequences of $[0,\ell(0)]$. A diffeomorphism $f\in\Diff^1_+([0,\ell(0)])$ is called $(\textbf{\texttt{a}},\textbf{\texttt{b}})$-{\em diffeomorphism\/} if $f(a_n)=b_n$ for all $n\in\Z$. The family of $(\textbf{\texttt{a}},\textbf{\texttt{b}})$-diffeomorphisms of $[0,\ell(0)]$ is denoted by $\Diff^1_+(\ell;0;\textbf{\texttt{a}},\textbf{\texttt{b}})$. This is a closed (and nonempty) subspace of $\Diff^1_+([0,\ell(0)])$ and therefore it is completely metrizable.

Observe that an $(\textbf{\texttt{a}},\textbf{\texttt{b}})$-diffeomorphism is tangent at the identity in its extreme points, this is a direct consequence of the definition of synchronized sequences.

Let $\cZ^{1\ast}_+(\ell;0;\textbf{\texttt{a}},\textbf{\texttt{b}})$ denote the subfamily of homeomorphisms $c\in\cZ^{1\ast}_+(\ell;0)$ such that $c_{|[0,\ell(0)]}\in \Diff^1_+(\ell;0;\textbf{\texttt{a}},\textbf{\texttt{b}})$.
\end{defn}

\begin{rem}
Let $\textbf{\texttt{c}}=\{c_n\}_{n\in\Z}$ be another increasing bisequence synchronized with $\textbf{\texttt{a}}$ and $\textbf{\texttt{b}}$. Observe that, if $c\in \cZ^{1\ast}_+(\ell;0;\textbf{\texttt{a}},\textbf{\texttt{b}})$ and $d\in \cZ^{1\ast}_+(\ell;0;\textbf{\texttt{b}},\textbf{\texttt{c}})$, then $d\circ c\in\cZ^{1\ast}_ +(\ell;0;\textbf{\texttt{a}},\textbf{\texttt{c}})$.
\end{rem}

In \cite[Proof of Proposition 4.5]{Ghys} it is shown that, for all $r\geq 0$, there exists a residual set $\Omega\subset \Diff^r(S^1)\times\Diff^r(S^1)\times S^1$ so that the action of the group $\langle f,g\rangle$ on $x$ is free for all $(f,g,x)\in\Omega$, i.e., only the identity fixes some point in the orbit of $x$ for that group. This proof can be readily adapted to our set up in order to show that our described minimal actions are generically faithful and free in restriction to a residual set.

\begin{lemma}\label{l:generic 2}
For any triple of synchronized bisequences $\textbf{\texttt{a}},\textbf{\texttt{b}}, \textbf{\texttt{c}}$, there exists a residual set $\Omega\subset \cZ^{1\ast}_+(\ell;0;\textbf{\texttt{a}},\textbf{\texttt{b}})\times \cZ^{1\ast}_+(\ell;0;\textbf{\texttt{b}},\textbf{\texttt{c}})\times S^1$ such that for every 
$(c,d,x)\in\Omega$ the action of $\rho_{c,d}$ in the orbit of $x$ is free, in particular $\rho_{c,d}$ is faithful.
\end{lemma}
\begin{proof}
Since each factor is completely metrizable, it follows that $\cZ^{1\ast}_+(\ell;0;\textbf{\texttt{a}},\textbf{\texttt{b}})\times \cZ^{1\ast}_+(\ell;0;\textbf{\texttt{b}},\textbf{\texttt{c}})\times S^1$ is a Baire space. For $w\in \Gamma_{\ell}^{\Join}$, set $w_{c,d}=\rho_{c,d}({w})^{-1}$ (see Notation~\ref{n:notation 1}). Let us consider the sets
\[
X_w=\{(c,d,x)\in \cZ^{1\ast}_+(\ell;0;\textbf{\texttt{a}},\textbf{\texttt{b}})\times \cZ^{1\ast}_+(\ell;0;\textbf{\texttt{b}},\textbf{\texttt{c}})\times S^1\mid\ w_{c,d}(x)=x\}
\] for 
$w\in \Gamma_\ell^{\Join}\setminus\{\id\}$. It is clear that each $X_w$ is closed; therefore, if they also have empty interior, it will follow that 
\[
\Omega= \left(\cZ^{1\ast}_+(\ell;0;\textbf{\texttt{a}},\textbf{\texttt{b}})\times \cZ^{1\ast}_+(\ell;0;\textbf{\texttt{b}},\textbf{\texttt{c}})\times S^1\right)\setminus \bigcup_{w} X_w
\] 
is a residual set satisfying the required conditions.

Let us assume that $X_w$ has nonempty interior for some $w$. Let $v$ be a minimal element of $\Gamma^{\Join}_\ell\setminus\{\id\}$ relative to its length in normal form $|v|$ so that $X_v$ has nonempty interior.

Observe that the restriction of $\rho_{c,d}$ to each $G^\star$ is conjugated to the original projective action of $G^\star$, these are faithful actions by construction and their elements do not admit identity germs at any point by analyticity. It follows that $|v|>1$.

Let
\[\varphi_w: \cZ^{1\ast}_+(\ell;0;\textbf{\texttt{a}},\textbf{\texttt{b}})\times \cZ^{1\ast}_+(\ell;0;\textbf{\texttt{b}},\textbf{\texttt{c}})\times S^1\to\cZ^{1\ast}_+(\ell;0;\textbf{\texttt{a}},\textbf{\texttt{b}})\times \cZ^{1\ast}_+(\ell;0;\textbf{\texttt{b}},\textbf{\texttt{c}})\times S^1
\]
be the homeomorphism defined by $\varphi_w(c,d,x)=(c,d,w^{-1}_{c,d}(x))$. Therefore $\varphi_{w}(X_{u})$ is closed with empty interior for any $u$ with $|u|<|v|$ and any $w\in \Gamma_\ell^{\Join}\setminus\{\id\}$. It is also clear that 

\[\varphi_{w}(X_u)=\{(c,d,x)\in \cZ^{1\ast}_+(\ell;0;\textbf{\texttt{a}},\textbf{\texttt{b}})\times \cZ^{1\ast}_+(\ell;0;\textbf{\texttt{b}},\textbf{\texttt{c}})\times S^1\mid\ u_{c,d}\circ w_{c,d}(x)=w_{c,d}(x))\}\;.\] 

Let $U$ be a nonempty open subset of $X_v$ and choose $( c, d, x)\in U$ so that it does not belong to any $\varphi_{w}(X_{u})$ for $0<|u|<|v|$. Such an element exists because countably many sets with empty interior cannot cover $U$ in a Baire space.
\begin{afirm}\label{cl:property}
Choose $z\in S^1$ arbitrary. There exists $(c,d,x)\in U$ such that  $(c,d,x)\notin \varphi_w(X_u)$, for  $0<|u|<|v|$ and $x$ does not belong to the $\rho_{c,d}$ orbit of $z$.
\end{afirm}
\begin{proof}[Proof of Claim~\ref{cl:property}]
Define, for each $w\in\Gamma^{\Join}_{\ell}$, $$Y_{z,w}=\{(e,f,w_{e,f}(z))\mid (e,f)\in\cZ^{1\ast}_+(\ell;0;\textbf{\texttt{a}},\textbf{\texttt{b}})\times \cZ^{1\ast}_+(\ell;0;\textbf{\texttt{b}},\textbf{\texttt{c}})\}$$ and $Y_z=\bigcup_w Y_{z,w}$. It is clear that each $Y_{z,w}$ is closed with empty interior. Therefore $Y_z$ is meager and we can choose $(c,d,x)\in U$ such that $(c,d,x)$ does not belong to $Y_z$ nor $\varphi_w(X_u)$ for any $w\in\Gamma^{\Join}_\ell$, $0<|u|<|v|$ .
\end{proof}

Let $v=s_{1}\cdots s_{|v|}$ be a normal representation of $v$. Thus
\[v_{c,d}= f_{|v|}\circ\cdots\circ f_2\circ f_1\,,\] where $f_i =\rho_{c,d}(s_i)^{-1}$ for all $i$.

Let $x_0=x$ and define recursively $x_i=f_i(x_{i-1})$ for $1< i\leq |v|$. Since $(c,d, x)\in X_v$, it follows that $x_{|v|}=x_0$.

Without loss of generality we can assume also that $x_i\neq\infty$  for all $i$ and $0\neq\overline{x}_i\neq b_n$ for all $n\in\Z$. This comes from the previous Claim~\ref{cl:property} applied to $\infty$, $0$ and $b_n$, $n\in\N$.

For simplicity, let us denote $k_i=k_{x_i}$ (defined in Notation~\ref{n:notation 1}). We claim that $\overline{x}_0,\overline{x}_1\dots,\overline{x}_{|v|-1}$ are pairwise different. Otherwise, if $\ell^{k_{i}}(x_i)=\ell^{k_{j}}(x_j)$ for $i<j$, then we can take $u=s_{i+1}\cdots s_{j-1}\cdot(s_j\cdot \ell^{k_j-k_i})$ and $w=s_1\cdots s_i$. These are also normal representations and therefore $|w|$ and $|u|$ are strictly lower than $|v|$. Thus $u_{c,d}(w_{c,d}(x))=(w_{c,d}(x))$ and this contradicts the choice of $( c, d, x)$.

Thus every $\overline{x}_i$, $0\leq i\leq |v|-1$, belongs to $]0,\ell(0)[\,\setminus\{b_n\mid n\in\Z\}$ and are pairwise distinct. Recall that $f_{|v|}\left(x_{|v|-1}\right)=x_{|v|}=x_0$.
 
By means of Lemma~\ref{l:perturbation} there exists $\varphi\in\cZ_+^{1\ast}(\ell;0)$, where $\varphi_{|[0,\ell(0)]}$ is supported in a neighborhood of $x_{|v|-1}$ in $]0,\ell(0)[\,\setminus\{b_n\mid\ n\in\Z\}$ that does not contain any $\overline{x}_i$ for $0\leq i<|v|-1$ and such that $\varphi^{-1}\circ f_{|v|}\circ \varphi(x_{|v|-1})\neq x_{|v|}$. Observe also that $\varphi(b_n)=b_n$ for all $n$ and therefore $\varphi\in \cZ^{1\ast}_+(\ell;0;\textbf{\texttt{b}},\textbf{\texttt{b}})$.

If $s_{|v|}\in G^\triangleright$ then, we choose the perturbation: 
\[e=\varphi^{-1}\circ c\,,\ f=d\;.\]
otherwise choose 
\[e:= c\,,\ f:=d\circ\varphi\;;\] 

Since $\varphi\in\cZ^{1\ast}_+(\ell;0;\textbf{\texttt{b}},\textbf{\texttt{b}})$, it follows that $(e,f)\in \cZ^{1\ast}_+(\ell;0;\textbf{\texttt{a}},\textbf{\texttt{b}})\times \cZ^{1\ast}_+(\ell;0;\textbf{\texttt{b}},\textbf{\texttt{c}})$.

Let $v_{e,f}=g_{|v|}\circ\cdots g_2\circ g_1$, where $g_i =\rho_{e,f}(s_i)^{-1}$ for all $i$. Set $y_0=x$ and define recursively $y_i=g_i(y_{i-1})$. It follows that $y_i = x_i$ for all $0\leq i <|v|$ and thus \[v_{e,f}(x)=g_{|v|}(x_{|v|-1})=\varphi^{-1}\circ f_{|v|}\circ \varphi (x_{|v|-1})\neq x_{|v|}=x.\]

If $\varphi$ is sufficiently close to $\id$, we get $(e,f,x)\in U$ contradicting that $U\subset X_v$.
\end{proof}

We can control the topology of the leaves in $\FF^{\Join}_{c,d}$ associated to the $\rho_{c,d}$-orbits of the points $b_n$, $n\in\Z$, for a generic choice of $c,d$. This is equivalent to control their stabilizers. Observe that the previous Lemma implies that, for a generic choice of the pair $(c,d)$, the generic leaf of $\FF^{\Join}_{c,d}$ is homeomorphic to a plane.

\begin{defn}[Predefined stabilizers]
Let $w\in \Gamma^{\Join}_\ell$, let $w=s_1\cdots s_n$ be a normal representation of $w$ (see Remark~\ref{r:amalgalm}) and let $\texttt{b}=\{b_k\}_{k\in\Z}$ be an increasing bisequence in $]0,\ell(0)[$ with accumulation points in $0, \ell(0)$. Suppose that $w_{c,d}(b_m)=b_m$ for some $m\in\Z$, and let $w_{c,d}=f_{n}\circ\cdots\circ f_1$ where $f_i=\rho_{c,d}({s_i})^{-1}$. Set $x_0=b_m$ and define inductively $x_i=f_{i}(x_{i-1})$ for $1\leq i\leq n$, we obtain a cycle of points $x_0,\dots,x_n$, where $x_0=x_n$. We say that the cycle is {\em predefined} if, for every $i\geq 1$, there exists $j_i\in\Z$ such that $\overline{x}_i=b_{j_i}$.

If this cycle associated to a normal representation of $w$ is {\em predefined}, it will be said that $w$ is a {\em predefined stabilizer} of $b_m$. Let $\Stab_{c,d}^\textsc{P}(\texttt{b},m)$ denote the set of predefined stabilizers associated to the action $\rho_{c,d}$ at $b_m$ for $c,d\in \cZ^{1\ast}_+(\ell;0;\textbf{\texttt{a}},\textbf{\texttt{b}})\times \cZ^{1\ast}_+(\ell;0;\textbf{\texttt{b}},\textbf{\texttt{c}})$  
\end{defn}

\begin{prop}
The definition of predefined stabilizer does not depend on the normal representation of $w$ and $\Stab_{c,d}^\textsc{P}(\texttt{b},m)$ is a subgroup of $\Gamma_\ell^{\Join}$ for all $m\in\Z$.
\end{prop}
\begin{proof}
Let $p_1 \cdots p_n=w=q_1\cdots q_n$ be two normal representations of $w$. From Remark~\ref{r:amalgalm} it follows that there exists integers $n_i\in \Z$, $0\leq i\leq n$, where $n_0=0$, such that $p_i=\ell^{-n_{i-1}}q_i\cdot \ell^{n_i}$ for all $1\leq i\leq n$ . Set $x_0=b_m=y_0$ and define $x_i=\rho_{c,d}(p_i)^{-1}(x_{i-1})$ and $y_i=\rho_{c,d}(q_i)^{-1}(y_{i-1})$, $i\geq 1$. It follows that $y_i=\ell^{n_i}(x_i)$ for all $i\geq 1$ and therefore each pair $x_i,y_i$ belong to the same $\ell$-orbit, it follows that the $x_i$'s form a predefined cycle if and only if the $y_i$'s also do.

Let us show now that $\Stab_{c,d}^\textsc{P}(\texttt{b},m)$ is a group for all $m\in\Z$. It is clear that it contains $\id$. It is also clear that it is closed under inverses since the inversion of a normal representation is still a normal representation. It just remains to show that it is closed under compositition.

Let $w,v\in \Stab_{c,d}^\textsc{P}(\texttt{b},m)$ and let $w=p_1\cdots p_r$, $v=q_1\cdots q_s$ be normal representations of $w$ and $v$. Let $x_0,\dots,x_r$ and $y_0,\dots, y_s$ be the predefined cycles associated to these normal representations. If $p_1\cdots p_r\cdot q_1\cdots q_s$ is still a normal representation, then it is clear that $w\cdot v$ will be a predefined stabilizer, where the associated predefined cycle is $x_0,\dots,x_r (= y_0),y_1,\dots,y_s$. In general $p_1\cdots p_r\cdot q_1\cdots q_s$ does not need to be a normal representation for $w\cdot v$, and this only occurs when $p_r$ and $q_1$ belong to the same free factor, say $G^\triangleright$. If $p_r\cdot q_1$ is not a power of $\ell$, then setting $h_1=p_r\cdot q_1$, we shall obtain the normal representation $p_1\cdots p_{r-1}\cdot h_1\cdot q_2\cdots q_s$ which induces the following predefined cycle: $x_0,\dots,x_{r-1}, y_1,\dots,y_s$, so $w\cdot v$ is still a predefined stabilizer. If $h_1=p_r\cdot q_1$ is a power of $\ell$, then we obtain that $h_2=p_{r-1}\cdot h_1\cdot q_2\in G^\triangleleft$. If $h_2$ is not a power of $\ell$, then we shall obtain the normal representation  $p_1\cdots p_{r-2}\cdot h_2\cdot q_3\cdots q_s$ with the following associated predefined cycle: $x_0,\dots,x_{r-2}, y_2,\dots,y_s$. This process of reduction to a normal representation is finite and, in any case, the associated cycle is predefined, thus $w\cdot v\in\Stab_{c,d}^\textsc{P}(\texttt{b},m)$ as desired.
\end{proof}

\begin{prop}\label{p:well defined stabilizers}
The definition of predefined stabilizer does not depend on $c,d$, i.e., for every $(c, d)$ and $(e,f)$ in $\cZ^{1\ast}_+(\ell;0;\textbf{\texttt{a}},\textbf{\texttt{b}})\times \cZ^{1\ast}_+(\ell;0;\textbf{\texttt{b}},\textbf{\texttt{c}})$, we have $\Stab_{c,d}^\textsc{P}(\texttt{b},m)= \Stab_{e,f}^\textsc{P}(\texttt{b},m)$.
\end{prop}
\begin{proof}
Assume that $w\in\Stab_{c,d}^\textsc{P}(\texttt{b},m)$. Let $w=s_1\cdots s_n$ be a normal representation and let $x_0,\dots, x_n$ be the associated predefined cycle. This means that for all $0\leq i\leq n$ there exists $k_i,j_i\in\Z$ such that $x_i= \ell^{k_i}(b_{j_i})$, where $x_0=b_m=x_n$.

Let $(e,f)\in\cZ^{1\ast}_+(\ell;0;\textbf{\texttt{a}},\textbf{\texttt{b}})\times \cZ^{1\ast}_+(\ell;0;\textbf{\texttt{b}},\textbf{\texttt{c}})$. Set $y_0=b_m$ and set $y_i=\rho_{e,f}(s_i)^{-1}(y_{i-1})$ for $1\leq i\leq n$. Let $g_i=s_i^{-1}$ for all $1\leq i\leq n$, these are projective diffeomorphisms. If $s_i\in G^\triangleright$ (resp. $G^\triangleleft$), then $\rho_{c,d}(s_i)^{-1}= c \circ g_i \circ c^{-1}$ (resp. $=d^{-1} g_i d$) and $\rho_{e,f}(s_i)^{-1}= e\circ g_i\circ e^{-1}$ (resp. $=f^{-1}\circ g_i\circ f$). 

By definition $x_0=b_m=y_0$, we shall show that, in fact $x_i=y_i$ for all $i\in\{0,\dots n\}$. Assume, as induction hypothesis, that $x_{i-1}=y_{i-1}$ for some $i$, then $y_{i-1}=\ell^{k_{i-1}}(b_{j_{i-1}})$ for suitable integers $k_{i-1}$ and $j_{i-1}$. Assume without loss of generality that $s_i\in G^\triangleright$. It follows that
\begin{align*}
y_i&=e\circ g_i \circ e^{-1} (\ell^{k_{i-1}}(b_{j_{i-1}}))=e\circ g_i\circ \ell^{k_{i-1}}(e^{-1}(b_{j_{i-1}}))=\\
e\circ& g_i\circ \ell^{k_{i-1}}(c^{-1}(b_{j_{i-1}}))= e\circ c^{-1}\circ c\circ g_i\circ c^{-1}(x_{i-1})=\\
e\circ& c^{-1}(x_i) = e\circ c^{-1} (\ell^{k_i}(b_{j_i}))= \ell^{k_i}\circ e\circ c^{-1} (b_{j_i})= \ell^{k_i}(b_{j_i})=x_i\;.
\end{align*}
Here we use that $e,c$ commute with $\ell$ and $e^{-1}(b_j)=c^{-1}(b_j)$ for all $j\in\Z$. An analogous reasoning applies when $s_i\in G^\triangleleft$ and thus, by finite induction $y_i=x_i$ for all $i$. Therefore the cycle $y_i$, $0\leq i\leq n$ is predefined for any pair $(e,f)$ as desired.
\end{proof}

\begin{notation}
By means of the above Proposition, we can use the notation $\Stab^{\textsc P}_{m}(\ell;0;\textbf{\texttt{a}};\textbf{\texttt{b}};\textbf{\texttt{c}})$ in order to denote the  group of predefined stabilizers $\Stab^\textsc{P}_{c,d}(\textbf{\texttt{b}},m)$ associated to any pair $(c,d)\in \cZ^{1\ast}_+(\ell;0;\textbf{\texttt{a}},\textbf{\texttt{b}})\times \cZ^{1\ast}_+(\ell;0;\textbf{\texttt{b}},\textbf{\texttt{c}})$ 
\end{notation}

\begin{lemma}\label{l:generic 2.5}
Let $\textbf{\texttt{a}}=\{a_n\}_{n\in\Z}$,$\textbf{\texttt{b}}=\{b_n\}_{n\in\Z}$, $\textbf{\texttt{c}}=\{c_n\}_{n\in\Z}$  be synchronized bisequences. Assume that all the points $a_n, c_n$, $n\in\Z$, do not belong to the orbits of $\infty$ or $0$ for the projective actions of the groups $G^\triangleright$ and $G^\triangleleft$ on the circle. There exists a residual set $\Delta$ of $\cZ^{1\ast}_+(\ell;0;\textbf{\texttt{a}},\textbf{\texttt{b}})\times \cZ^{1\ast}_+(\ell;0;\textbf{\texttt{b}},\textbf{\texttt{c}})$ so that the $\rho_{c,d}$-orbits of each $b_m$ contains neither $0$ nor $\infty$ for all $(c,d)\in\Delta$.
\end{lemma}
\begin{proof}
Let $\Delta_n$ be the set of pairs $(c,d)\in\cZ^{1\ast}_+(\ell;0;\textbf{\texttt{a}},\textbf{\texttt{b}})\times \cZ^{1\ast}_+(\ell;0;\textbf{\texttt{b}},\textbf{\texttt{c}})$ such that $\infty\neq w_{c,d}(b_m)\neq 0$ for all $w\in \Gamma_\ell^{\Join}$, with $|w|=n$ and all $m\in\Z$. Let us prove that each $\Delta_n$ is residual. 

Observe first that $\Delta_1=\cZ^{1\ast}_+(\ell;0;\textbf{\texttt{a}},\textbf{\texttt{b}})\times \cZ^{1\ast}_+(\ell;0;\textbf{\texttt{b}},\textbf{\texttt{c}})$ since every $w$ with normal length $1$ satisfies that $w_{c,d}$ is conjugated to an element of $G$ by $c^{-1}$ or $d$ (this only depends on what factor of the amalgamated product contains $w$). This conjugation preserves the $\infty$ point and maps the sequence $b_n$ to the sequences $a_n$ or $c_n$ (corresponding to $c^{-1}$ or $d$ respectively). It follows that $w_{c,d}(b_m)$ is the infinity point if and only if $\infty$ belongs to the $G^\triangleright$-orbit of some $a_n$ or the $G^\triangleleft$-orbit of $c_n$ in contradiction with the choice of these sequences, the same applies to $0$.

Now we proceed by induction. Assume that $\Delta_n$ is residual for all $n < k$ and let us show that $\Delta_k$ is also residual. Let $w=s_1\cdots s_k$ be a normal representation of an element with normal length equal to $k$, thus $w_{c,d}=f_k\circ\cdots \circ f_1$ where $f_i=\rho_{c,d}(s_i)^{-1}$. Assume that $w_{c,d}(b_m)=\infty$ (resp. $0$) for some $(c,d)\in\bigcap_{i=1}^{k-1}\Delta_i$.

Let $x_0=b_m$ and define, recursively, $x_i= f_i(x_{i-1})$. By the induction hypothesis $x_i$ cannot be $\infty$ or $0$ for every $0\leq i<k$. Moreover $x_i$ cannot be an element in the $\ell$-orbit of any $b_j$, $j\in\Z$ and $0<i<k$, for otherwise, if $x_i=\ell^{r}(b_j)$ for some $j,r\in\Z$ and $i>0$, then we can take $u=(\ell^{-r}\cdot s_{i+1})\cdot s_{i+2}\cdots s_k$. This is still a normal representation of $u$ and $|u|<k$. But $u_{c,d}$ maps $b_j$ to $\infty$ (resp. $0$) contradicting the choice of $c,d$. Thus $\overline{x}_i\in\,]0,\ell(0)[\,\setminus\{b_n\mid\ n\in\Z\}$ for all $1\leq i<k$.

It is also clear that $\overline{x}_i\neq\overline{x}_j$ for $i\neq j\in\{1,\dots, k-1\}$. Otherwise, assuming $i<j$ and $x_j=\ell^k(x_i)$, take $v=s_1\cdots \cdot s_{i-1}\cdot (s_{i}\cdot \ell^{-r})\cdot s_{j+1}\cdots s_{k}$. It follows that $|v|<k$ and $w_{c,d}(b_m)=\infty$ (resp. $=0$). This is a contradiction again with the choice of $(c,d)$.

If $s_k\in G^\triangleright$ (resp. $s_k\in G^\triangleleft$), then we can perturb $c$ (resp. $d$) as in Lemma~\ref{l:generic 2}, conjugating by a flow $\varphi=\phi_t\in\cZ^{1\ast}_+(\ell;0)$, for any $t\neq 0$, whose restriction to $[0,\ell(0)]$ is supported in a small neighborhood of $\overline{x}_{k-1}$ with empty intersection with $\{b_n\mid\ n\in\Z\}\cup\{\overline{x}_1,\dots,\overline{x}_{k-2}\}$, in particular $\phi_t\in\cZ^{1\ast}_+(\ell;0;\textbf{\texttt{b}},\textbf{\texttt{b}})$. The perturbed action $\rho_{e,f}$ for $e=\varphi^{-1}\circ c$, $f=d$ (resp. $e=c$, $f=d\circ\varphi$) satisfies that $w_{e,f}(b_m)\neq \infty$ (resp. $0$).

This shows that the sets 
\[F_w^{m}=\{(c,d)\in\cZ^{1\ast}_+(\ell;0;\textbf{\texttt{a}},\textbf{\texttt{b}})\times \cZ^{1\ast}_+(\ell;0;\textbf{\texttt{b}},\textbf{\texttt{c}})\mid\ w_{c,d}(b_m)=\infty\ \text{or}\ w_{c,d}(b_m)=0\}\] have empty interior. Since each $F^m_w$ is closed, it follows that $\bigcup_{m\in\Z}F_w^m$ is meager, denote this meager set by $F_w$. Therefore $\Delta_n=\cZ^{1\ast}_+(\ell;0;\textbf{\texttt{a}},\textbf{\texttt{b}})\times \cZ^{1\ast}_+(\ell;0;\textbf{\texttt{b}},\textbf{\texttt{c}})\setminus\bigcup_{|w|<n} F_w$ is residual as desired. Finally, $\Delta=\bigcap_n\Delta_n$ is also residual and satisfies the required conditions.
\end{proof}

Now we are ready to prove our main Lemma: for a generic choice of actions $\rho_{c,d}$, the group of stabilizers of $\rho_{c,d}$ at each $b_m$ (see  Definition~\ref{d:stabilizer_group}) coincides with the group of predefined stabilizers of $b_m$.

\begin{lemma}\label{l:generic 3}
Let $\textbf{\texttt{a}}=\{a_n\}_{n\in\Z},\textbf{\texttt{b}}=\{b_n\}_{n\in\Z}, \textbf{\texttt{c}}=\{c_n\}_{n\in\Z}$ be synchronized bisequences. Assume that all the points $a_n, c_n$, $n\in\Z$, do not belong to the orbits of $\infty$ or $0$ for the projective actions of the groups $G^\triangleright$ and $G^\triangleleft$ on the circle. There exists a residual set $\Omega$ of $\cZ^{1\ast}_+(\ell;0;\textbf{\texttt{a}},\textbf{\texttt{b}})\times \cZ^{1\ast}_+(\ell;0;\textbf{\texttt{b}},\textbf{\texttt{c}})$ so that 
$\Stab_{\rho_{c,d}}(b_m)=\Stab^{\textsc {P}}_{m}(\ell;0;\textbf{\texttt{a}};\textbf{\texttt{b}};\textbf{\texttt{c}})$ for all $( c, d)\in \Omega$ and all $m\in\Z$.
\end{lemma}
\begin{proof}
For $u\in \Gamma_\ell^{\Join}$ and $m\in\Z$, let us define
\[
Y_{u}^m=\{(c,d)\in \cZ^{1\ast}_+(\ell;0;\textbf{\texttt{a}},\textbf{\texttt{b}})\times \cZ^{1\ast}_+(\ell;0;\textbf{\texttt{b}},\textbf{\texttt{c}})\mid\ u_{c,d}(b_m) = b_m\}
\]

We shall show that $Y^m_u$ has empty interior for all $u\notin \Stab_m^\textsc{P}(\ell;0;\textbf{\texttt{a}};\textbf{\texttt{b}};\textbf{\texttt{c}})$.  This is a countable family of closed sets, therefore its intersection will be meager. The residual complement will be formed by those pairs $(c,d)\in \cZ^{1\ast}_+(\ell;0;\textbf{\texttt{a}},\textbf{\texttt{b}})\times \cZ^{1\ast}_+(\ell;0;\textbf{\texttt{b}},\textbf{\texttt{c}})$ such that $\Stab_{\rho_{c,d}}(b_m)=\Stab^{\textsc {P}}_{m}(\ell;0;\textbf{\texttt{a}};\textbf{\texttt{b}};\textbf{\texttt{c}})$       as desired.

Assume that there exists $v\notin\Stab_{m}^\textsc{P}(\ell;0;\textbf{\texttt{a}};\textbf{\texttt{b}};\textbf{\texttt{c}})$ of minimal length in normal form, $|v|=r$, so that $Y_v^m$ has nonempty interior. Let $\Delta$ be the residual set defined in Lemma~\ref{l:generic 2.5}. Let $(c,d)\in\Delta$ be an interior point of $Y_v^m$.

Let $v=s_1\cdots s_r$ be a normal representation of $v$. Observe that $r>1$, since for $r=1$ usual stabilizers are also predefined stabilizers. Set $x_0=b_m$ and $x_i=\rho_{c,d}(s_i)^{-1}(x_{i-1})$ for $i\geq 1$. Since $(c, d)\in\Delta$ it follows that  $\infty\neq\bar{x}_i\neq 0$ for all $i$. Since $v\notin\Stab_{m}^\textsc{P}(\ell;0;\textbf{\texttt{a}};\textbf{\texttt{b}};\textbf{\texttt{c}})$ it follows that there exists $k_0\in\{1,\dots, r-1\}$ so that $\bar{x}_{k_0}\neq b_n$, for all $n\in\Z$. Let $k_0$ be the maximal integer as above such that $\overline{x}_{k_0}\neq \bar{x}_i$ for all $0\leq i< k_0$.

Take a flow $\phi_t\in\cZ^{1\ast}_+(\ell;0;\textbf{\texttt{b}},\textbf{\texttt{b}})$ close to the identity for all $t$ close to $0$ and such that $\phi_{|[0,\ell(0)]}$ is supported in a neighborhood of $\overline{x}_{k_0}$ that does not contain any $b_n$, $n\in\Z$, and no other $\overline{x}_i$ for $0\leq i< k_0$ (this is possible by means of Lemma~\ref{l:perturbation}  and Corollary~\ref{c:perturbation}).

Choose $t_0\neq 0$ arbitrary; if $s_{k_0+1}\in G^\triangleright$ (resp. $G^\triangleleft$), then set $e_0=\phi_{-t_0}\circ c$ and $f_0=d$ (resp. $e_0=c$, $f_0=d\circ\phi_{t_0}$). Let $y_0 = x_0=b_m$ and define $y_i=\rho_{e_0,f_0}(s_i)^{-1}(y_{i-1})$ for $i\geq 1$. It follows that $x_i=y_i$ for $0\leq i\leq k_0$ and $\overline{y}_{k_0+1}\neq \overline{y}_i$ for all $0\leq i\leq k_0$; moreover, for almost all $t_0$ we also have $\overline{y}_{k_0+1}\neq b_n$ for all $n\in\Z$. Observe that this works even when $\overline{x}_{k_0+1}=\overline{x}_{k_0}$ since non trivial elements of both $G^\triangleright$ and $G^\triangleleft$ are not germs of the identity at any point.

If $v_{e_0,d_0}(b_m)\neq b_m$, for instance when $k_0=r-1$, then this perturbation suffices for our purposes. Otherwise let $k_1>k_0$ be the maximal integer so that $\overline{y}_{k_1}\neq \overline{y}_i$ for $0\leq i < k_1$, it exists since $k_0+1$ satisfies this property. Moreover we can guarantee that $y_{k_1}\neq\infty$ and $\overline{y}_{k_1}$ does not meet $\ell(0)$ nor any $b_n$ $n\in\N$ since the $\ell$-orbit of $x_{k_0}$ does not meet these points and the perturbation can be chosen arbitrarily close to $\id$,  i.e. take $t_0$ sufficiently close to $0$. Set $z_0=y_0=b_m$. By the same argument as above, there exists a perturbed pair $(e_1,f_1)$ defined from $(e_0,f_0)$ and a suitable flow in $\Z^{1\ast}_+(\ell;0)$, so that the points $z_i=\rho_{e_1,f_1}(s_i)^{-1}(z_{i-1})$ satisfy that $\bar{z}_{k_1+1}\neq \bar{z}_{i}$ for $0\leq i\leq k_1+1$. 

We can repeat this process inductively until we get an integer $N\in\N$, $N\leq r$, such that $v_{e_N,f_N}(b_m)\neq b_m$. This contradicts the fact that $(c,d)$ is an interior point of $Y_v^m$.
\end{proof}

\section{Realizing the topologies}\label{s:realizing topology}

In this subsection we deal with the problem of the realization of any open oriented surface as the leaf of a suitable $\FF^{\Join}_{c,d}$ (see Definition~\ref{d:F_c-d}), 
the proof of Theorem~\ref{Main Theorem} will be a consequence of this construction. This is done in two steps:
\begin{enumerate}
\item The manifold $W_F^\nu$, constructed in Definition~\ref{d:realization W_F} appears as a proper submanifold with boundary of some leaf of
$\FF^{\Join}_{c,d}$ for all $(c,d)\in\cZ^{1\ast}_+(\ell;0;\textbf{\texttt{a}},\textbf{\texttt{b}})\times \cZ^{1\ast}_+(\ell;0;\textbf{\texttt{b}},\textbf{\texttt{b}})$, for two suitable 
synchronized bisequences $\textbf{\texttt{a}},\textbf{\texttt{b}}$ that depends in the tree $T_F$ and its coloring.

\item The second step is to show that, for a residual set in $\cZ^{1\ast}_+(\ell;0;\textbf{\texttt{a}},\textbf{\texttt{b}})\times \cZ^{1\ast}_+(\ell;0;\textbf{\texttt{b}},\textbf{\texttt{b}})$, the leaf which contains 
$W_F^\nu$ is homeomorphic to the complete filling of $W^\nu_F$, and therefore it is homeomorphic to $S_F^\nu$ by Proposition~\ref{p:W=S}.
\end{enumerate}

The case of the plane, even though it is not very interesting, is the generic case, in fact this is a corollary of Lemma~\ref{l:generic 2}.

\begin{notation}
Recall the definitions of $T^\star$ as a transverse circle of the parabolic tori associated to $\beta_\star$. Let us denote by $L^{\star}_x$ the leaf of $\FF^\star$ that meets the point $x\in T^\star$. Similarly, $B^{\star}_x$ will denote the boundary component of $L^{\star}_x$ that meets a point $x\in T^\star$ (this is an orbit of the trace foliation in the parabolic tori). In order to relax notations $\cL^i$ will denote, simultaneously, the family of leaves of type $i$ of both $\FF^\triangleright$ and $\FF^\triangleleft$. Finally, recall that $\FF^{\Join}_{c,d}$ denotes the foliation obtained from applying Dehn's fillings to the elliptic tori of $\FF_{c,d}$,  $T^{\Join}$ will denote the transverse circle of $\FF^{\Join}_{c,d}$ obtained by the identification of the $T^\star$'s and $L^{c,d}_x$ will denote the leaf of $\FF^{\Join}_{c,d}$ that meets a point $x\in T^{\Join}$.

For the sake of readability, given (bi)equences $a_n\in T^\triangleright$, $c_n\in T^\triangleleft$, $n\in\Z$, then  $L^\triangleright_{\bm{n}}$ (resp. $L^\triangleleft_{\bm{n}}$) will denote\footnote{The bold subscript is used to distinguish $L^\star_{\bm{n}}$ from $L^\star_n$ which just denotes the leaf passing through $n\in\Z\subset\R\cup\{\infty\}\equiv S^1$.} the leaf $L^\triangleright_{a_n}$ (resp. $L^\triangleleft_{c_n}$), for $n\in\Z$. The same convention will be applied for $B^\triangleright_{\bm{n}}$ and $B^\triangleleft_{\bm{n}}$, $n\in\Z$ .
\end{notation}

We also recall that for every $L\in\cL^{1}$, a leaf of type $1$ on $\FF^\star$, the sets $$B^+=\{x\in [0,\ell(0)[\,\mid B_x\ \text{is an outer component}\}$$ and 
$$B^-=\{x\in [0,\ell(0)[\,\mid B_x\ \text{is an inner component}\}$$ are dense in $[0,\ell(0)[$, this is a direct consequence of Lemma~\ref{l:dense out inn}.

\begin{rem}\label{r:synchronization}
The first point will be the construction of suitable synchronized bisequences $a_n$ and $b_n$, $n\in\Z$. We shall see that there is a lot of freedom in the given choices. Let $b_n$ be an arbitrary increasing bisequence whose accumulation points are $0$ and $\ell(0)$. Let $$\varepsilon_n = \min\left\{\dfrac{b_n}{n},\dfrac{(\ell(0)-b_n)}{n}, \dfrac{b_{n+1}-b_n}{2},\dfrac{b_n-b_{n-1}}{2}\right\}\;,$$ whenever we choose points $a_n,c_n\in (b_n-\varepsilon_n,b_n+\varepsilon_n)$ the resulting bisequences will be syncronized with $b_n$.
\end{rem}

Recall that surfaces with one end are obtained as complete fillings (or the interior) of the manifolds $H^0_k$ defined in Subsection~\ref{ss:one end surfaces}.

\begin{prop}\label{p:one end realization}
Let $k\in \N\cup\{\infty\}$. There exist synchronized bisequences $\textbf{\texttt{a}},\textbf{\texttt{b}},\textbf{\texttt{c}}$ such that the leaf $L^{c,d}_{b_0}$ of $\FF^{\Join}_{c,d}$ contains a proper\footnote{More precisely: the submanifold is a closed subset or, equivalently, the inclusion is a proper map.} submanifold homeomorphic to $H_k^0$ for all $(c,d)\in \cZ^{1\ast}_+(\ell;0;\textbf{\texttt{a}},\textbf{\texttt{b}})\times \cZ^{1\ast}_+(\ell;0;\textbf{\texttt{b}},\textbf{\texttt{c}})$.
\end{prop}

\begin{proof}
Let $\textbf{\texttt{b}}=\{b_n\}_{n\in\Z}$ be an arbitrary increasing bisequence accumulating to both $0$ and $\ell(0)$, we shall consider this sequence naturally embedded in both $T^\triangleright$ and $T^\triangleleft$. Define $\varepsilon_n$ as in Remark~\ref{r:synchronization}. We shall define, inductively, increasing bisequences $a_n\in T^\triangleright$ and $c_n\in T^\triangleleft$ such that $|a_n-b_n|<\varepsilon_n$ and  $|c_n-b_n|<\varepsilon_n$. This guarantees the synchronization of the three bisequences.

Let $k\in\N$. By minimality of outer boundary components, for each $0\leq i\leq k-1$, there exists $a_{3i}\in T^\triangleright$ so that $L^{\triangleright}_{\bm{3i}}\in \cL^1$ and $B^\triangleright_{\bm{3i}}$ is an outer boundary component. By minimality of inner boundary components, there exists $a_{3i+1}\in T^\triangleright$ so that $B^\triangleright_{\bm{3i+1}}$ is an inner boundary component of $L^{\triangleright}_{\bm{3i}}$.

Choose now $c_{3i}, c_{3i+1}\in T^\triangleleft$ so that $L^{\triangleleft}_{\bm{3i}}\in\cL^0$ and $c_{3i+1}\in L^\triangleleft_{\bm{3i}}$. Observe that attaching the boundary components $B^\triangleright_{\bm{3i}}\to B^\triangleleft_{\bm{3i}}$ and $B^\triangleright_{\bm{3i+1}}\to B^\triangleleft_{\bm{3i+1}}$ of the leaves $L^\triangleright_{\bm{3i}}$ and $L^\triangleleft_{\bm{3i}}$ produce a space homeomorphic to the noncompact tile $H^0$. Finally, choose $a_{3i+2}\in L^\triangleright_{\bm{3i}}$ and $c_{3i+2}\in L^\triangleleft_{\bm{3i+1}}$.

The leaves $L^\triangleright_{\bm{3i}}$ and $L^\triangleleft_{\bm{3i}}$ are always chosen to be pairwise different, i.e. $L^\triangleright_{\bm{3i}}\neq L^\triangleright_{\bm{3j}}$ and $L^\triangleleft_{\bm{3i}}\neq L^\triangleleft_{\bm{3j}}$ for all $i\neq j$. All these choices are possible by the density of (outer and inner) bounday components of any leaf in $\FF^\star$ and the fact that there are infinitely many leaves of type $0$ and $1$ in our foliated blocks.

The rest of elements of the bisequences which were not previously defined can be rather arbitrary, but always verifying the synchronization condition and belonging to leaves of $\FF^\star$ not considered in the previous construction. Let $\textbf{\texttt{a}},\textbf{\texttt{c}}$ denote, respectively, the bisequences $a_n$ and $c_n$.

For every $(c,d)\in \cZ^{1\ast}_+(\ell;0,\textbf{\texttt{a}},\textbf{\texttt{b}})\times \cZ^{1\ast}_+(\ell;0,\textbf{\texttt{b}},\textbf{\texttt{c}})$ we have that $d\circ c (a_j)=c_j$, thus the respective boundary components are attached in $\FF^{\Join}_{c,d}$ for any choice of $(c,d)\in\cZ^{1\ast}_+(\ell;0;\textbf{\texttt{a}},\textbf{\texttt{b}})\times \cZ^{1\ast}_+(\ell;0;\textbf{\texttt{b}},\textbf{\texttt{c}})$. All of these handles are embedded in the leaf $L^{c,d}_{b_0}$ of $\FF^{\Join}_{c,d}$ since each $L^\triangleright_{\bm{3i}}$ has a boundary component attached with a boundary component of $L^\triangleleft_{\bm{3i+1}}$ via the mapping $a_{3i+2}\to c_{3i+2}$.

This construction defines a proper map $\jmath: H^0_k\to L^{c,d}_{b_0}$ from $H^0_k$ to the leaf $L^{c,d}_{b_0}$  of $\FF_{c,d}^{\Join}$. This may fail to be an embedding at the boundary of $H^0_k$ for some choices of $(c,d)$ since the boundary components not identified by the bisequences $\textbf{\texttt{a}}$ and $\textbf{\texttt{c}}$ are uncontrolled and they might be identified by the attaching map $d\circ c$. In this case, we can remove from of $H^0_k$ a tubular neighborhood of the boundary components, the restriction of $\jmath$ to the resulting subset is the desired embedding.

Note that the previous construction also works when $k=\infty$, since the sequences $a_n$ and $c_n$ can be defined inductively for any $n\in\N$.
\end{proof}

The case of noncompact oriented surfaces with more than one end will follow a similar scheme. As in section~\ref{s:realization}, let $T_F$ be a subtree of the binary tree with no dead ends, containing the root element. Consider the orientation given in Notation~\ref{n:step construction} and let $\nu: V_F\to\{0,1\}$ be a vertex coloring.

Recall the definition of the noncompact manifold $W^\nu_F$ given in Definition~\ref{d:realization W_F}. Those steps of construction must be happening on the transverse gluing of the parabolic tori of $\FF^\triangleright$ and $\FF^\triangleleft$.  

\begin{notation}\label{n:enum_condition}We shall enumerate $V_F$ according with oriented levels, i.e., if $v_n$ is in a lower level than $v_m$, then $n\leq m$, here we shall consider $v_0=\bm{\mathring{v}}$ to be the root element. This enumeration implies that, for all $n\geq 0$ (resp. $n\leq 0$) $v_{n}$ is the target (resp. origin) of just one oriented edge whose origin (resp. target) is $v_{k(n)}$ for some $0\leq k(n)<n$ (resp. $n<k(n)\leq 0$). If $\deg(v_n)=3$, then it is the origin (resp. target) of exactly two oriented edges, in this case we shall define $i(n)$ as the minimum (resp. maximum) index so that $v_{i(n)}$ is the target (resp. origin) of an edge with origin (resp. target) in $v_n$ and $q(n)$ as the maximum (resp. minimum) index with this property. It is clear that $|n|<|i(n)|<|q(n)|$.
\end{notation}

Recall that in the construction of $W^\nu_F$ (see Definition~\ref{d:realization W_F}), when $\deg(v_n)=3$ then the noncompact tile $N_{v_n}$ must be attached with $N_{v_{i(n)}}$ and $N_{v_{q(n)}}$ by two boundary components that cannot be simulteneusly inner nor outer. We shall assume that $N_{v_n}$ is attached to $N_{v_{i(n)}}$ (resp. $N_{v_{q(n)}}$) by an inner (resp. outer) component.

\begin{prop}\label{p:F realization}
For any conected subtree $T_F$ with no dead ends of the binary tree and any coloring $\nu: V_F\to\{0,1\}$ there exist synchronized bisequences $\textbf{\texttt{a}},\textbf{\texttt{b}},\textbf{\texttt{c}}$ such that the leaf $L^{c,d}_{b_0}$ of $\FF^{\Join}_{c,d}$ contains a closed submanifold homeomorphic to $W^\nu_F$ for all $(c,d)\in \cZ^{1\ast}_+(\textbf{\texttt{a}},\textbf{\texttt{b}})\times \cZ^{1\ast}_+(\textbf{\texttt{b}},\textbf{\texttt{c}})$.
\end{prop}
\begin{proof}
The proof is inductive but there are several cases to consider, for the sake of readability we organize the proof into several parts.
\proofpart{1}{The root vertex}
Choose first $a_0,a_{-1}\in T^\triangleright$ belonging to the same leaf $L^{\triangleright}_{\bm{0}}\in\cL^1$ (i.e., $L^{\triangleright}_{\bm{0}}=L^{\triangleright}_{\bm{1}}$) and such that $B^\triangleright_{\bm{0}}$ is outer and $B^\triangleright_{\bm{-1}}$ is inner. 

If $\nu(v_0)=0$, then choose $c_0, c_{-1}\in T^\triangleleft$ such that $L^{\triangleleft}_{\bm{0}}, L^\triangleleft_{\bm{-1}}\in\cL^0$ are different leaves. The boundary union of these leaves via the identifications $B^\triangleright_{\bm{i}}\to B^\triangleleft_{\bm{i}}$, $i=-1,0$, is a space homeomorphic to the noncompact tile $P^1$. This is, in this case, the definition of $N_{\bm{\mathring{v}}}$ in Definition~\ref{d:realization W_F}.

If $\nu(v_0)=1$, then choose $c_0,c_{-1}\in T^\triangleleft$ belonging to the same leaf $L^{\triangleleft}_{\bm{0}}\in\cL^1$ and such that $B^\triangleleft_{\bm{0}}$ is outer and $B^\triangleleft_{\bm{-1}}$ is inner. The boundary union of these leaves via the identifications $B^\triangleright_{\bm{i}}\to B^\triangleleft_{\bm{i}}$, $i=-1,0$, is a space homeomorphic to the noncompact tile $H^1$. This is, in this case, the definition of  $N_{\bm{\mathring{v}}}$ in Definition~\ref{d:realization W_F}.

This is the case $0$ of an induction process that will now be depicted.

\proofpart{2}{The induction scheme}
Let $n\in\N\cup\{0\}$, assume that there exists $K_n\in\N\cup\{0\}$ such that $a_k$, $c_k$ were defined for $-1\leq k\leq K_n$ and the boundary gluings $B^\triangleright_{\bm{k}}\to B^\triangleleft_{\bm{k}}$, for $-1\leq k\leq K_n$, applied to the leaves $L^\triangleright_{\bm{k}}$'s and $L^\triangleleft_{\bm{k}}$'s, defines a manifold homeomorphic to the quotient manifold of $\bigsqcup_{i=0}^n N_{v_i}$ by the equivalence relation given in Definition~\ref{d:realization W_F}. 

We shall make two assumptions, called A1 and A2 in our induction hypothesis. Although they are not necessary, they will make the proof simpler.
\begin{itemize}
\item[A1:] for each $0\leq i\leq n$ there exists a partition $\{I_{i,n}\mid\ 0\leq i\leq n\}$ of $\{-1,\dots, K_n\}$ so that the identifications $B^\triangleright_{\bm{j}}\to B^\triangleleft_{\bm{j}}$ between the leaves $L^\triangleright_{\bm{j}}$'s  and $L^\triangleleft_{\bm{j}}$'s, for $j\in I_{i,n}$, define a manifold homeomorphic to the corresponding noncompact tile $N_{v_i}$.

\item[A2:] If $\deg(v_i)=3$, then we shall assume also that there exists $j(i,n)\in I_{i,n}$ such that $L^\triangleright_{\bm{j(i)}}$ is a leaf of type $1$ in $\FF^\triangleright$. If $\deg(v_i)=2$, then set $j(i,n)$ as the maximum of $I_{i,n}$.
\end{itemize}

For the sake of readability, let $K$, $I_i$ and $j(i)$ denote respectively $K_n$, $I_{i,n}$ and $j(i,n)$.

\proofpart{3}{The bridge to the case $n+1$}

Let $e$ be the (unique) edge such that $v_{n+1}=t(e)$ and set $v_m =o(e)$ for a unique $0\leq m\leq n$.

Choose $a_{K+1}\in T^\triangleright$ such that $a_{K+1}\in L^\triangleright_{\bm{j(m)}}$ and $c_{K+1}\in T^\triangleleft$ such that $L^\triangleleft_{\bm{K+1}}\in\cL^0$ is a leaf different from $L^\triangleleft_{\bm{i}}$ for $-1\leq i\leq K$. 

If $\deg(v_m)=3$, then $N_{v_m}$ is homeomorphic to the noncompact tile $P^1$ or $H^1$ (depending on its color). If $n+1=q(m)$ (resp. $n+1=i(m)$), then choose $a_{K+1}$ such that $B^\triangleright_{\bm{K+1}}$ is an outer (resp. inner) boundary component of $L^\triangleright_{\bm{j(m)}}$ (they exists by assumption A2). This choice of $a_{K+1}$, $c_{K+1}$ do not modify the topology produced in the previous steps since $L^\triangleleft_{\bm{K+1}}$ is a leaf of type $0$, this can be considered as a topological bridge to the next inductive step.

\proofpart{4}{The case where $\nu(v_{n+1})=0$}
If $\nu(v_{n+1})=0$ and $\deg(v_{n+1})=2$ (resp. $\deg(v_{n+1})=3)$, then choose $a_{K+2}$ and $c_{K+2}$  such that $c_{K+2}\in L^\triangleleft_{\bm{K+1}}$ and $L^\triangleright_{\bm{K+2}}\in\cL^0$ (resp. $L^\triangleright_{\bm{K+2}}\in\cL^1$) is different from the leaves used in the previous steps of the induction.

Observe that the boundary union of $L^\triangleright_{\bm{K+2}}$ with $L^\triangleleft_{\bm{K+2}}=L^\triangleleft_{\bm{K+1}}$ (identifying the boundary components $B^\triangleright_{\bm{K+2}}$ with $B^\triangleleft_{\bm{K+2}}$) is homeomorphic to a leaf of type $0$ (resp. type $1$). That is exactly the definition of $N_{v_{n+1}}$ in this case and this is joined with the construction performed in the previous steps via the identification $B^\triangleright_{\bm{K+1}}\to B^\triangleleft_{\bm{K+1}}$ (the bridge). 

It follows, in this case, that $K_{n+1}=K_n+2$, $I_{n+1,n+1}=\{K_n+2\}$, $I_{n,n+1}=I_{n,n}\cup\{K+1\}$ and $I_{i,n+1}=I_{i,n}$ for $0\leq i < n$, these sets clearly form a partition of $\{-1,\dots,K+2\}$ and satisfy assumptions A1 and A2.

\proofpart{5}{The case where $\nu(v_{n+1})=1$}
Choose first $c_{K+2}\in L^\triangleleft_{\bm{K+1}}$ and $a_{K+2}\in T^\triangleright$ such that $L^\triangleright_{\bm{K+2}}\in\cL^1$ is a leaf different from the used in the previous inductive steps and $B^\triangleright_{\bm{K+2}}$ is inner.\\

If $\deg(v_{n+1})=2$, then choose $a_{K+3},c_{K+3}$ so that $a_{K+3}\in L^\triangleright_{\bm{K+2}}$, $c_{K+3}\in L^\triangleleft_{\bm{K+1}}$  and $B^\triangleright_{\bm{K+3}}$ is outer. Observe that the identifications $B^\triangleright_{\bm{i}}\to B^\triangleleft_{\bm{i}}$, for $K+1\leq i\leq K+3$, defines a manifold homeomorphic to $H^0$ and is attached to the construction performed in the previous steps via the bridge $L^\triangleleft_{\bm{K+1}}$. 

In this case $K_{n+1}=K_n+3$, $I_{n+1,n+1}=\{K_n+2,K_n+3\}$, $I_{n,n+1}=I_{n,n}\cup\{K+1\}$, and $I_{i,n+1}=I_{i,n}$ for $0\leq i<n$. These sets clearly form a partition of $\{-1,\dots,K+3\}$ and satisfy assumptions A1 and A2.\\

If $\deg(v_{n+1})=3$, then choose $a_{K+3}, a_{K+4}\in T^\triangleright$ such that $a_{K+3}, a_{K+4}\in L^\triangleright_{\bm{K+2}}$. Choose now $c_{K+3},c_{K+4}\in T^\triangleleft$ such that $L^\triangleleft_{\bm{K+3}}\in\cL^1$ is a leaf different from the used in the previous inductive steps and $c_{K+4}\in L^\triangleleft_{\bm{K+3}}$. Choose these points so that $B^\triangleright_{\bm{K+3}}$, $B^\triangleleft_{\bm{K+3}}$  are both inner and $B^\triangleright_{\bm{K+4}}$, $B^\triangleleft_{\bm{K+4}}$ are both outer components. 

This is equivalent to attaching a copy of the noncompact tile $H^1$ (see Figure~\ref{f:H1}) to the construction given in the previous inductive steps via the bridge $L^\triangleleft_{\bm{K+1}}$.

It this last case $K_{n+1}=K_n + 4$, $I_{n+1,n+1}=\{K+2,K+3, K+4\}$, $I_{n,n+1}=I_{n,n}\cup\{K+1\}$ and $I_{i,n+1}=I_{i,n}$ for $0\leq i<n$. These sets clearly form a partition of $\{-1,\dots,K+4\}$ and satisfy assumptions A1 and A2.

Recall that the choice of $a_k$ and $c_k$ can be made in a synchronized way with $b_k$ as was observed in Remark~\ref{r:synchronization}. This follows from the minimality of (outer and inner) boundary components and the fact that leaves of type $0$ and $1$ are infinite in any foliated block.  Observe also that $a_n$ and $c_n$ can also  be chosen so that $L^\star_{\bm{n}}$ does not meet $0$, this will be important for a future reference.

\proofpart{6}{End of the proof}
Beginning with the root element $v_0$, this construction can be performed inductively in $n$. This process defines increasing sequences $a_k, c_k$ for $k\geq -1$ synchronized with the chosen $b_k$ (see Remark~\ref{r:synchronization}). The identifications of the boundary components $B^\triangleright_{\bm{k}}\to B^\triangleleft_{\bm{k}}$ reproduces the identifications of the noncompact tiles $N_{v_n}$, $n\geq 0$, in the construction of the manifold $W^\nu_F$.

A similar argument can be applied to $n<-1$, just by interchanging the words outer and inner, target and origin, $K$ by $-K$, $+$ by $-$ and choosing the new leaves in $\FF^\triangleright$ and $\FF^\triangleleft$ different from those chosen in the first inductive process. Let $\textbf{\texttt{a}},\textbf{\texttt{c}}$ denote, respectively, the bisequences $a_n$ and $c_n$.

For every $(c,d)\in \cZ^{1\ast}_+(\ell;0,\textbf{\texttt{a}},\textbf{\texttt{b}})\times \cZ^{1\ast}_+(\ell;0,\textbf{\texttt{b}},\textbf{\texttt{c}})$ we have that $d\circ c (a_j)=c_j$ and therefore the above inductive construction of $W^\nu_F$ defines a proper map $\jmath: W^\nu_F\to L^{c,d}_{b_0}$ from $W^\nu_F$ to the leaf $L^{c,d}_{b_0}$  of $\FF_{c,d}^{\Join}$. This may fail to be an embedding at the boundary of $W^\nu_F$ for some choices of $(c,d)$ since the boundary componetns not identified by the bisequences $\textbf{\texttt{a}}$ and $\textbf{\texttt{c}}$ are uncontrolled and might be identified by the attaching map $d\circ c$. In this case, we can remove a tubular neighborhood of the boundary components of $W^\nu_F$ and the restriction of $\jmath$ to the resulting subset will define the desired embedding.
\end{proof}

\begin{prop}\label{p:realization}
Let $S$ be an noncompact oriented surface. There exist synchronized increasing bisequences $\textbf{\texttt{a}}$, $\textbf{\texttt{b}}$ and $\textbf{\texttt{c}}$ in $S^1$ such that there exists a residual set $\Omega_S\subset \cZ^{1\ast}_+(\ell;0,\textbf{\texttt{a}},\textbf{\texttt{b}})\times \cZ^{1\ast}_+(\ell;0,\textbf{\texttt{b}},\textbf{\texttt{c}})$ where the 
leaf $L^{c,d}_{b_0}$ of $\FF^{\Join}_{c,d}$ meeting $b_0$ is homeomorphic to $S$ for every $(c,d)\in\Omega_S$.
\end{prop}
\begin{proof}
Let us take first an arbitrary increasing bisequence $\textbf{\texttt{b}}$ accumultaing to $0$ and $\ell(0)$.

If $S$ has just one end, then  $\textbf{\texttt{a}}$ and $\textbf{\texttt{c}}$ as the bisequences constructed in Proposition~\ref{p:one end realization} so that $H^0_k$ is properly embedded in $L^{c,d}_{b_0}$ for all $(c,d)\in\cZ^{1\ast}_+(\textbf{\texttt{a}},\textbf{\texttt{b}})\times \cZ^{1\ast}_+(\textbf{\texttt{b}},\textbf{\texttt{c}})$, where $k$ is the genus of $S$ (possibly infinite). If $S$ has more than one end, then set $T_F$ be a subtree of the binary tree with no dead ends and $\nu$ a vertex coloring such that $S$ is homeomorphic to $S^\nu_F$. Let $\textbf{\texttt{a}}$ and $\textbf{\texttt{c}}$ be the bisequences constructed in Proposition~\ref{p:F realization} so that $W^\nu_F$ is properly embedded in $L^{c,d}_{b_0}$ for all $(c,d)\in\cZ^{1\ast}_+(\textbf{\texttt{a}},\textbf{\texttt{b}})\times \cZ^{1\ast}_+(\textbf{\texttt{b}},\textbf{\texttt{c}})$. 

 The bisequences $\textbf{\texttt{a}}=\{a_n\}_{n\in\Z}$ and $\textbf{\texttt{c}}=\{c_n\}_{n\in\Z}$ were chosen so that the leaves $L^\star_{\bm{n}}$ do not meet the point $0\in T^\star$ for all $n\in\Z$. This implies that $a_n$ and $c_n$ do not belong to the orbit of $0$ for the projective actions of $G^\star$ in $T^\star$, moreover $a_n$ and $c_n$ do not belong to the $G^\star$-orbits of $\infty$ since the leaves $L^\star_{\bm{n}}$ are never of type $2$.


Let $\Omega_S$ be the residual set given by Lemma~\ref{l:generic 3} such that $\Stab_{\rho_{c,d}}(a_0)=\Stab^\textsc{P}_{b_0}(\ell;0,\textbf{\texttt{a}};\textbf{\texttt{b}};\textbf{\texttt{c}})$ for all $(c,d)\in\Omega_S$.  Let $W$  denote $H^0_k$ or $W^\nu_F$ and let $\jmath: W\to L^{c,d}_{b_0}$ be the proper map defined in Propositions~\ref{p:one end realization} and \ref{p:F realization}.

Observe that for any $(c,d)\in\Omega_S$ there are no more stabilizers than the predefined ones.  

\begin{afirm}\label{cl:predefined filling}
Each boundary component of $\jmath(W)$ is attached with a single boundary component of some leaf of type $0$ in some $\FF^\star$ different from the leaves $L^\star_{\bm{n}}$, $n\in\Z$. Moreover, leaves attached with $\jmath(W)$ cannot be attached to each other.
\end{afirm}

\textit{Proof of Claim}~\ref{cl:predefined filling}

 Let $W^\star = \bigcup_{m\in\Z} L^\star_{\bm{m}}$, observe that $\jmath(W)$ can be obtained as the boundary union of $W^\triangleright$ with $W^\triangleleft$ via $d\circ c$. 

Each boundary component of $\jmath(W)$ defines two boundary components of leaves in the foliated projective blocks $\FF^\star$, one of the form $B^\triangleright_{c^{-1}(x)}$ and other of the form $B^\triangleleft_{d(x)}$ for some $x\in S^1$ such that $\bar{x}$ is not an element of the sequence $\textbf{\texttt{b}}$. 

One of these leaves, $L^\triangleright_{c^{-1}(x)}$ or $L^\triangleleft_{d(x)}$, is included in some $W^\star$.

Suppose without loss of generality that $L^\triangleright_{c^{-1}(x)}$ is included in $W^\triangleright$, then $L^\triangleleft_{d(x)}$ must be a leaf of type $0$. Otherwise, a leaf of type $1$ is in correspondence with a projective map that fixes $d(x)$ and therefore, modulo conjugation, it would define a  stabilizer of $b_0$ in $\Gamma^{\Join}_\ell$ (for the $\rho_{c,d}$ action) that is not predefined since $\bar{x}\notin \{b_n\mid\ n\in\Z\}$. This is in contradiction with the fact that $(c,d)\in\Omega_S$.

Moreover $L^\triangleleft_{d(x)}$ cannot be attached to $W^\triangleright$ by more than one boundary component, as this would define an stabilizer of $b_0$ in $\Gamma^{\Join}_\ell$ whose associated cycle of points contains the point $x$, hence it is not a predefined stabilizer, again a contradiction. By the same reason $L^\triangleleft_{d(x)}$ cannot be included in $W^\triangleleft$ nor attached to any leaf in $\FF^\triangleright$ that is also attached to $W^\triangleleft$. This completes the proof of Claim~\ref{cl:predefined filling}.

The argument given in Claim~\ref{cl:predefined filling} can be repeated inductively: set $W_0=W$ and set $W_n^\triangleright$ (resp. $W_n^\triangleleft$) be the union of $W_{n-1}^\triangleright$ (resp. $W_{n-1}^\triangleleft$) with the leaves in $\FF^\triangleright$ (resp. $\FF^\triangleleft$) which are attached to $W_{n-1}^\triangleright$ (resp. $W_{n-1}^\triangleleft$) by some boundary component, then the leaves of $\FF^\star$ attached to $W_n^\star$ are new leaves of type $0$ that are attached by just one boundary component and they do not attach to each other. This is exactly the process depicted in Remark~\ref{r: plane}.

Observe that the leaf $L^{c,d}_{b_0}$ is homeomorphic to the boundary union of $\bigcup_{n\in\N} W_n^\triangleright$ with $\bigcup_{n\in\N} W_n^\triangleleft$.
It follows that the leaf passing through $b_0$ is homeomorphic to the complete filling of $W$ (that is also homeomorphic to its interior). Since $\Intr(W)$ is homeomorphic to $S$ (see Proposition~\ref{p:W=S}) it follows that $L^{c,d}_{b_0}$ is homeomorphic to $S$.
\end{proof}

As a consequence of the proof of Proposition~\ref{p:realization} we can complete the proof of the main Theorem~\ref{Main Theorem}.

\begin{proof}[Proof of Theorem~\ref{Main Theorem}]
Let $\sg$ be a separating loop in the typical orbifold $O$ and let $\FF^\star=\overline{\FF}(G^\star,O^\star,h^\star)$ be foliated projective blocks for $\star\in\{\triangleright,\triangleleft\}$. Let $S_n$, $n\in\N$, be any countable family of noncompact oriented surfaces. Let $\{b_n\}_{n\in\Z}$ by any increasing bisequence with limit points $0$ and $\ell(0)$. Let $\xi: \N\times\Z\to\Z$ any bijection so that $\xi(n,k)\leq\xi(n,m)$ if and only if $k\leq m$. For each $n\in \N$ define $b_{n,k}=b_{\xi(n,k)}$. For all $n\in \N$, $\{b_{n,k}\}_{k\in\Z}$ is a subsequence of $b_n$ and therefore is also increasing and accumulates to the same limit points as $b_n$.

By Propositions~\ref{p:one end realization} and \ref{p:F realization}, for each $n\in\N$ there exist two bisequences $\textbf{\texttt{a}}_n=a_{n,k}$ and $\textbf{\texttt{c}}_n=c_{n,k}$, $k\in\Z$, synchronized with $\{b_{n,k}\}_{k\in\Z}$ and a residual set $\Omega_n\subset \cZ^{1\ast}_+(\textbf{\texttt{a}}_n,\textbf{\texttt{b}}_n)\times \cZ^{1\ast}_+(\textbf{\texttt{b}}_n,\textbf{\texttt{c}}_n)$ provided by Proposition~\ref{p:realization} so that the leaf $S_n$ is homeomorphic to the leaf $L^{c,d}_{b_{n,0}}$ of $\FF^{\Join}_{c,d}$ for all $(c,d)\in\Omega_n$. Leaves of $\FF^\triangleright$ and $\FF^\triangleleft$ passing through points in different bisequences must be 
chosen to be different, i.e., if $n\neq m$, then $a_{n,j}$ and $c_{n,j}$ does not belong to any $L^\triangleright_{a_{m,k}}$, $L^\triangleleft_{c_{m,k}}$ for all $j,k\in\Z$.

Let $\textbf{\texttt{a}}$ and $\textbf{\texttt{c}}$ be the bisequences defined by $a_m=a_{\xi^{-1}(m)}$ and $c_m=b_{\xi^{-1}(m)}$, respectively, obtained from gathering together the bisequences $\textbf{\texttt{a}}_n$ and $\textbf{\texttt{c}}_n$ defined above. These are still synchronized with $b_m$ if the $a_{n,k}$'s and $c_{n,k}$'s are chosen sufficiently close to $b_{n,k}$ (this always can be done by minimality following Remark~\ref{r:synchronization}).

By Lemma~\ref{l:generic 3} there exists a residual set $\Omega\subset \cZ^{1\ast}_+(\textbf{\texttt{a}},\textbf{\texttt{b}})\times \cZ^{1\ast}_+(\textbf{\texttt{b}},\textbf{\texttt{c}})$ where $\Stab_{\rho_{c,d}}(b_m)=\Stab^\textsc{P}_{b_m}(\textbf{\texttt{a}};\textbf{\texttt{b}};\textbf{\texttt{c}})$ for all $(c,d)\in\Omega$. The assumption that leaves associated to different points of the bisequences are different implies that $\Stab^\textsc{P}_{b_{n,k}}(\textbf{\texttt{a}};\textbf{\texttt{b}};\textbf{\texttt{c}})=\Stab^\textsc{P}_{b_{n,k}}(\textbf{\texttt{a}}_n;\textbf{\texttt{b}}_n;\textbf{\texttt{c}}_n)$. Therefore the leaf $L^{c,d}_{b_{n,0}}$ is homeomorphic to $S_n$ for all $(c,d)\in\Omega$ as desired.

The ambient manifold of $\FF^{\Join}_{\id,\id}$ is Seifert and homeomorphic to $M(O,\sg)$ (see Remark~\ref{r:ambient manifold}), whose basis orbifold is typical and homeomorphic (as an orbifold) with $O$. It follows, since $c,d$ are close to $\id$, that the ambient manifold of $\FF^{\Join}_{c,d}$ is also $M(O,\sg)$ for such a choice of $c,d$.
\end{proof}

\section{Open questions}

It is an interestng question if there exists a critical regularity for the coexistence of finite and infinite geometric types. More precisely

\begin{quest}
Can surfaces of finite and infinite geometric type coexist in a transversely $C^k$ codimension one minimal foliation on a compact $3$-manifold for some $k\geq 1$? If possible, what topologies can coexist?
\end{quest}

The answer to the previous question is true for $k=0$ and also in the category of bi-Lipschitz homeomorphisms. It is unclear if our examples can be regularized to $C^1$, i.e., if there exists a $C^0$ conjugation between some of our examples and a $C^1$-foliation. Regarding this question we want to note that the centralizer of a generic $C^1$ diffeomorphism is trivial \cite{Bonatti-Crovisier-Wilkinson} but there exist $C^1$-diffeomorphisms with large\footnote{Large in the sense that it contains flows supported in small neighborhoods, this is our main tool for our generic results.} centralizers \cite{Navas}. For $C^k$, $k\geq 2$, this is impossible by the Kopell Lemma.

Another question is related to the possible ambient manifolds. All our examples are foliations on Seifert manifolds and transverse to the fibers, in every case they are finitely covered by suspensions over hyperbolic surfaces (see Remark~\ref{r:Selberg}). It seems also possible to export our construction to some graph manifolds. This is left for a future work.

\begin{quest}
Let $M$ be a fixed closed $3$-manifold. Can every noncompact orientable surface be homeomorphic to a leaf of a minimal foliation on $M$? Can arbitrary topologies of finite and infinite geometric type coexist as leaves of a minimal foliation of $M$? If not, give precise obstructions.
\end{quest}

Observe that the previous question includes the case of non-oriented surfaces, which were not treated in this work. Recall that leaves of hyperbolic foliations must be orientable. It seems that our construction can be improved to obtain a minimal foliation with an arbitrary noncompact nonorientable surface as one of its leaves, in this case the ambient Seifert manifold can be nonorientable. A first case of interest is to understand the hyperbolic non-typical orbifolds, for instance those where the incompressible loop $\sigma$ does not disconnect the orbifold, this leads to extend our results from amalgamated products to HNN extensions. We avoid the case of HNN extensions in this work since, in this case, $h(\beta_\triangleright)$ and $h(\beta_\triangleleft)^{-1}$ are never projectively conjugated, and this would lead to a nonpreserving orientation gluing map between the parabolic tori. This more general situation will be studied in a forthcoming work dealing with the realization of nonorientable noncompact surfaces.

In \cite{ABMPW} a minimal hyperbolic lamination is given on a compact space where all the topologies of oriented surfaces appear as leaves of that lamination. We show that any countable family of topologies can coexist in a codimension one minimal hyperbolic foliation of a $3$-manifold, but it is unclear if our method of proof can be improved to realize all the topologies of noncompact surfaces in the same foliation.

\begin{quest}
Does there exist a minimal codimension one hyperbolic foliation on a closed $3$-manifold such that every noncompact orientable surface is homeomorphic to a leaf of that foliation?
\end{quest}

The last question is related to the case where the foliation is given by the suspension of a group action. In our examples, when $O$ is a hyperbolic surface, the Euler number is always even but this could be an artifact from our method of construction. Recall also that, when the Euler number is maximal, the leaves of minimal foliations must be planes or cylinders \cite[Th\'eor\'eme 3]{Ghys2}. So it is a natural question if there exist more obstructions for rich leaf topology depending on the Euler number of the ambient Seifert manifold.

\begin{quest}
Let $\Sigma$ be a closed surface of genus $g\geq 2$ and let $0\leq n <\chi(\Sigma)$. Can every noncompact oriented surface $S$ be homeomorphic to a leaf of a foliation given by a suspension of an action $\rho_S: \pi_1(\Sigma)\to\Homeo_+(S^1)$ with $|\eu(\rho_S)| = n$? If true, what topologies can coexist in the same foliation?
\end{quest}

\end{document}